\documentclass[oneside,english]{amsart}
\usepackage[T1]{fontenc}
\usepackage[latin9]{inputenc}
\usepackage{amsthm}
\usepackage{amstext}
\usepackage{amssymb}

\makeatletter
\numberwithin{equation}{section}
\numberwithin{figure}{section}
  \theoremstyle{plain}
  \newtheorem*{thm*}{Theorem}
\theoremstyle{plain}
\newtheorem{thm}{Theorem}
  \theoremstyle{definition}
  \newtheorem{defn}[thm]{Definition}
  \theoremstyle{plain}
  \newtheorem{prop}[thm]{Proposition}
  \theoremstyle{plain}
  \newtheorem{lem}[thm]{Lemma}
  \theoremstyle{remark}
  \newtheorem{rem}[thm]{Remark}

\makeatother

\usepackage{babel}

\begin{document}

\title{Differential Structure and Flow equations on Rough Path Space}

\author{Zhongmin Qian, Jan Tudor}


\begin{abstract}
We introduce a differential structure for the space of weakly geometric
$p$ rough paths over a Banach space $V$ for $2<p<3$. We begin by
considering a certain natural family of smooth rough paths and differentiating
in the truncated tensor series. The resulting object has a clear interpretation,
even for non-smooth rough paths, which we take to be an element of
the tangent space. We can associate it uniquely to an equivalence
class of curves, with equivalence defined by our differential structure.
Thus, for a functional on rough path space, we can define the derivative
in a tangent direction analogous to defining the derivative in a Cameron-Martin
direction of a functional on Wiener space. Our tangent space contains
many more directions than the Cameron-Martin space and we do not require
quasi-invariance of Wiener measure. In addition we also locally (globally)
solve the associated flow equation for a class of vector fields satisfying
a local (global) Lipshitz type condition. 
\end{abstract}
\maketitle

\section{Introduction}

The main examples of continuous random models are those constructed
by solving Itô's stochastic differential equations. By means of Itô's
integration, one is able to define a unique strong solution to the
following Stratonovich type of differential equation\begin{equation}
dX_{t}^{i}=f_{0}^{i}(t,X_{t})dt+\sum_{j=1}^{d}f_{j}^{i}(t,X_{t})\circ dW_{t}^{j}\text{, }X_{0}=x\label{r01}\end{equation}
where $i$ runs from $1$ to $n$, $W=(W^{1},\cdots,W^{d})$ is a
$d$-dimensional Brownian motion on a probability space, and $\circ d$
denotes the Stratonovich differential. Equation (\ref{r01}) has to
be interpreted as an integration equation\begin{equation}
X_{t}^{i}=x^{i}+\int_{0}^{t}f_{0}^{i}(s,X_{s})ds+\sum_{j=1}^{d}\int_{0}^{t}f_{j}^{i}(s,X_{s})\circ dW_{s}^{j}\text{, }X_{0}=x\label{r02}\end{equation}
where the integration is understood as the Stratonovich integrals
which in turn can be converted to Itô's integrals. Suppose that the
coefficients $f_{j}^{i}$ are smooth with bounded derivatives. The
important nature is that the strong solution of (\ref{r01}) is defined
only almost surely, although the distribution of $X=(X_{t})_{t\geq0}$
is determined uniquely and is independent of the Brownian motion $W$.
On the other hand, there is a measurable mapping $F$ from $R^{+}\times R^{n}\times\mathcal{C}(R^{+};R^{d})$
to $R^{n}$ associated with (\ref{r01}) such that $X_{t}=F(t,x,W)$
is the unique strong solution to (\ref{r01}). Moreover, for each
$t\geq0$ and $\omega\in\mathcal{C}(R^{+};R^{d})$, $x\rightarrow F(t,x,\omega)$
is a diffeomorphism of $R^{n}$. In particular, the strong solution
to (\ref{r01}) is differentiable in the initial data $x$, which
will not be surprising to anyone who has experience with dynamical
systems. It was Malliavin who first observed that the mapping $\omega\rightarrow F(t,x,\omega)$
is differential in direction $h$ which belongs to the Cameron-Martin
space of the Wiener measure, i.e. for $h\in H_{0}^{1}(R^{+};R^{d})$,
where $H_{0}^{1}(R^{+},R^{d})$ is the space of all paths $h$ in
$R^{d}$ whose generalized derivative $\dot{h}\in L^{2}(R^{+};R^{d})$.

In this article with the help of Lyons' continuity theorem we identify
the differential structure on the space of rough paths which allows
us to differentiate Wiener functionals along more tangent directions
than those determined by the Cameron-Martin space. One is thus able
to use the machinery of rough paths together with nonlinear functional
analysis to study Wiener functionals, providing powerful mathematical
tools.

In Malliavin's calculus, the Wiener functionals we are interested
in are functions on the space of continuous paths $C\left(\left[0,\infty\right);V\right)$,
where $V=R^{d}$ for simplicity. The distribution $\mu$ of the standard
Brownian motion in $R^{d}$ is a probability measure on $C\left(\left[0,\infty\right);V\right)$.
If $h\in C\left(\left[0,\infty\right);V\right)$ then the measurable
transformation $\tau_{h}$ which sends a path $x$ to $x+h$ gives
rise to a push-forward measure $\mu_{h}$ defined by $\mu_{h}\left(A\right)=\mu\circ\tau_{h}\left(A\right)$.
A classical result in probability theory says that $\mu_{h}$ is absolutely
continuous with respect to $\mu$ if and only if $h$ belongs to the
Cameron-Martin space $H$ consisting of all paths $h\in C\left(\left[0,\infty\right);V\right)$
whose generalized derivative $\dot{h}\in L^{2}([0,\infty))$. Moreover,
according to Cameron-Martin \cite{CamMart}, in this case, \[
\frac{d\mu_{h}}{d\mu}=\exp\left[\int_{0}^{\infty}\dot{h}(t)d\omega(t)-\int_{0}^{\infty}|\dot{h}(t)|^{2}dt\right]\]
where $d\omega(t)$ is understood as Itô's differential. This property
of the Wiener measure is known as the quasi-invariance of Wiener measure.

Malliavin (see \cite{MalDerivOrig}) initiated a study of differentiating
Wiener functionals on $C\left(\left[0,\infty\right);V\right)$ in
order to address the regularities of their laws. An important result
is that many Wiener functionals are smooth in the Cameron-Martin directions.
Because the Wiener functionals (namely solutions to some stochastic
differential equations) we are interested in are only defined almost
surely, it is possible to differentiate such functions on $C\left(\left[0,\infty\right);V\right)$
only along the directions given in the Cameron-Martin space, and therefore
one has to perturb a path in a Cameron-Martin direction in order to
preserve the measure. In rough path analysis the Wiener functionals
are lifted to continuous functions on rough path space and therefore
quasi-invariance of the Wiener measure is not required. This allows
us to develop a calculus of variations without referring to the Wiener
measure.

The theory of rough paths, see \cite{LyonsIbero} for a detailed discussion,
was motivated in part by a desire to have a deterministic or pathwise
way of dealing with stochastic differential equations. The core idea
is that for paths which have infinite variation as typical stochastic
paths do, for example Brownian motion, defining the integral as a
Riemann sum is not sufficient. It turns out that for less regular
paths, in addition to increments, one needs information about the
area enclosed by a path and possibly higher order volumes in order
to define an integration theory. The regularity of a rough path, in
general, determines how many higher order terms must be considered.
For simplicity, we restrict ourselves to the simplest true rough paths,
i.e. rough paths with roughness $p$ where $2<p<3$ (see below for
an explanation). A rough path $X$ with roughness $p$ (so called
a $p$-rough path) is a map on the simplex $\Delta_{T}:=\left\{ \left(s,t,\right):s,t\in\left[0,T\right]\right\} $
taking values in the truncated \ tensor algebra\[
T^{2}\left(V\right):=1\oplus V\oplus V^{\otimes2}\text{,}\]
which satisfies Chen's identity, $X_{s,t}\otimes X_{t,u}=X_{s,u}$
for all $s,t,u\in\left[0,T\right]$ with $s<t<u$, and a regularity
condition (\ref{eq-0e}). Here the tensor multiplication $\otimes$
takes place in $T^{2}\left(V\right)$ so that \begin{align*}
X_{s,u}^{1}= & X_{s,t}^{1}+X_{t,u}^{1}\text{,}\\
X_{s,u}^{2}= & X_{s,t}^{2}+X_{t,u}^{2}+X_{s,t}^{1}\otimes X_{t,u}^{1}\text{,}\end{align*}
where $X_{s,t}^{1}\in V$, $X_{s,t}^{2}\in V^{\otimes2}$ are the
components of $X_{s,t}$ in $V$ and $V^{\otimes2}$. $X$ has finite
$p$-variation in the sense that \begin{equation}
\sup_{\mathcal{D}}\left(\sum_{l}\left\vert X_{s,t}^{i}\right\vert ^{\frac{p}{i}}\right)^{\frac{i}{p}}<\infty\label{eq-0e}\end{equation}
for $i=1,2$.

Let $x(t)=X_{0t}^{1}$ for $t\leq T$. Then $X_{s,t}^{1}=x(t)-x(s)$.
We sometimes say $X$ is a rough path over the continuous path $x$.
On the other hand, if given a continuous path $x$ with finite variation
(up to time $T$), one may construct a rough path $X$, called the
canonical lift of $x$, by \begin{align*}
X_{s,t}^{1}= & x\left(t\right)-x\left(s\right)\\
X_{s,t}^{2}= & \int_{s<u_{1}<u_{2}<t}dx\left(u_{1}\right)\otimes dx\left(u_{2}\right)\end{align*}
where the integral is defined via Riemann sums. In this case Chen's
identity is just the additivity of iterated integrals over different
intervals. Such a $p$-rough path is called a \emph{smooth} $p$-rough
path.

The most interesting $p$-rough paths (where $2<p<3$) are of course
those over Brownian motion sample paths. Observe that Brownian motion
sample paths are, with probability one, Hölder continuous with exponent
less than one half, which implies they have finite $p$-variation
only for $p>2$. It is well established that almost all Brownian motion
sample paths can be lifted canonically to $p$-variation rough paths
for $2<p<3$.

The space of $p$-rough paths equipped with the $p$-variation distance
\[
d_{p}\left(X,Y\right)=\max_{i}\left[\sup_{\mathcal{D}}\left(\sum_{l}\left\vert X_{s,t}^{i}-Y_{s,t}^{i}\right\vert ^{\frac{p}{i}}\right)^{\frac{i}{p}}\right]\]
is a complete metric space, denoted by $\Omega_{p}\left(V\right)$.
This space contains two special subspaces $G\Omega_{p}\left(V\right)$
and $WG\Omega_{p}\left(V\right)$ which we define subsequently. The
$p$-rough paths which are the limit, in $p$-variation distance,
of a sequence of smooth rough paths are called geometric $p$-rough
paths and denoted $G\Omega_{p}\left(V\right)$. While weakly geometric
$p$-rough paths, denoted $WG\Omega_{p}\left(V\right)$, are the elements
of $\Omega_{p}\left(V\right)$ that can be realized as the limit in
uniform topology of the canonical lifts of bounded $p$-variation
smooth paths. See for example \cite{NoteOnGeoRP} for more details
on the differences between these spaces.

In this article, we identify a useful representation of the tangent
space associated to a natural differential structure on $WG\Omega_{p}\left(V\right)$.
The reason for our definition of derivative comes from the following
observation for a finite variation path $x$. Given another finite
variation path $y$, one can produce the variational path $x+\varepsilon y$
for $\varepsilon\in\left[0,1\right]$ say. This then induces a variation
at the level of rough paths, $X\left(\varepsilon\right)$, of the
canonical lift $X$ of $x$ which, due to the finite variation, is
given by \begin{align*}
X\left(\varepsilon\right)^{1}= & \int dx+\varepsilon\int dy\\
X\left(\varepsilon\right)^{2}= & \int dx\otimes dx+\varepsilon\left(\int dx\otimes dy+\int dy\otimes dx\right)+\varepsilon^{2}\int dy\otimes dy\end{align*}
where we have suppressed the limits of integration. In this case,
the derivative of $X\left(\varepsilon\right)$ in the linear space
$C\left(\Delta_{T},T^{2}\left(V\right)\right)$ is \[
\left.\frac{d}{d\varepsilon}\right\vert _{\varepsilon=0}X\left(\varepsilon\right)=\left(0,\int dy,\int dx\otimes dy+\int dy\otimes dx\right)\text{.}\]
Note that if $x$ and $y$ have finite $p$-variation, then the cross
iterated integrals $\int dx\otimes dy$ and $\int dy\otimes dx$ have
finite $\frac{p}{2}$ variation. However, in addition to varying the
increment, we can also vary the second level path independently, by
$\varphi$. Hence, we modify $X\left(\varepsilon\right)$ to include
both first and second level variations and obtain\begin{align*}
X\left(\varepsilon\right)^{1}= & \int dx+\varepsilon\int dy\\
X\left(\varepsilon\right)^{2}= & \int dx\otimes dx+\varepsilon\left(\int dx\otimes dy+\int dy\otimes dx+\varphi\right)+\varepsilon^{2}\int dy\otimes dy\end{align*}
so that \[
\left.\frac{d}{d\varepsilon}\right\vert _{\varepsilon=0}X\left(\varepsilon\right)=\left(0,\int dy,\int dx\otimes dy+\int dy\otimes dx+\varphi\right)\text{.}\]
We remark that, in some sense, $X\left(\varepsilon\right)$ is the
simplest variation of $X$ and that its derivative at $0$ can be
associated to the pair $\left(Z,\varphi\right)$ for $Z\in\Omega_{p}\left(V\oplus V\right)$
with $Z^{1}=\left(\int dx,\int dy\right)$, \[
Z^{2}=\left(\begin{array}{cc}
\int dx\otimes dx & \int dx\otimes dy\\
\int dy\otimes dx & \int dy\otimes dy\end{array}\right)\]
and $\varphi\in\Omega_{p/2}\left(V\oplus V\right)$. From this identification,
it is possible to make rigorous the meaning of the cross iterated
integrals $\int dx\otimes dy$ and $\int dy\otimes dx$ as certain
projections, denoted $\pi_{12}\left(Z\right)$ and $\pi_{21}\left(Z\right)$
(see below for an explanation), of an element of $Z\in\Omega_{p}\left(V\oplus V\right)$
even if $x$ and $y$ are non-finite variation paths. Hence, for each
pair $\left(Z,\varphi\right)$ we define a variational curve $V_{\left(Z,\varphi\right)}\left(\varepsilon\right)$
at $X$ by \begin{align*}
V_{\left(Z,\varphi\right)}\left(\varepsilon\right)^{1}= & X^{1}+\varepsilon\pi_{2}\left(Z\right)^{1}\\
V_{\left(Z,\varphi\right)}\left(\varepsilon\right)^{2}= & X^{2}+\varepsilon\left[\pi_{12}\left(Z\right)+\pi_{21}\left(Z\right)+\varphi\right]+\varepsilon^{2}\pi_{2}\left(Z\right)^{2}\end{align*}
where we use the notation \[
Z=\left(1,\left(\pi_{1}\left(Z\right)^{1},\pi_{2}\left(Z\right)^{2}\right),\left(\begin{array}{cc}
\pi_{1}\left(Z\right)^{2} & \pi_{1,2}\left(Z\right)\\
\pi_{2,1}\left(Z\right) & \pi_{2}\left(Z\right)^{2}\end{array}\right)\right)\text{,}\]
where $\pi_{i}$ and $\pi_{ij}$ are natural projections which should
be self-evident. Finally we have to identify the equivalence classes
of variations which give the same derivative. Therefore, we say that
$\left(Z,\varphi\right)$ is equivalent to $\left(\tilde{Z},\tilde{\varphi}\right)$
if \[
\left.\frac{d}{d\varepsilon}\right\vert _{\varepsilon=0}V_{\left(Z,\varphi\right)}\left(\varepsilon\right)=\left.\frac{d}{d\varepsilon}\right\vert _{\varepsilon=0}V_{\left(\tilde{Z},\tilde{\varphi}\right)}\left(\varepsilon\right)\]
and $\pi_{1}\left(Z\right)=X$ and denote the equivalence class $\left[Z,\varphi\right]$.
Note that we cannot uniquely assign a variational curve to an equivalence
class $\left[Z,\varphi\right]$ because we have a choice of the $\pi_{2}\left(Z\right)^{2}$
term. Our first main theorem \ref{thm:TgtFromCurve} shows that the
collection of all equivalence classes $[Z,\varphi]$ is the tangent
space at $X$, in the sense that, every possible differentiable curve
of rough paths starting at $X$ (whose derivative is taken in a function
space $C(\Delta,T^{(2)}(V))$) is determined uniquely by some $\left[Z,\varphi\right]$.

The idea behind the first main theorem can be described as the following.
If $X\left(\varepsilon\right)$ is an arbitrary curve of rough paths,
then the pair $\left(X^{1}\left(0\right),\left[\left.\frac{d}{d\varepsilon}\right\vert _{\varepsilon=0}X\left(\varepsilon\right)\right]^{1}\right)$
will in general not have a canonical lift to $\Omega_{p}\left(V\oplus V\right)$.
Therefore, a way of obtaining a $p$-rough path ($\left\lfloor p\right\rfloor =2$)
from just the increment level is required and is provided, though
not uniquely, by the Lyons-Victoir extension (see \cite{LyonsVicExten}).
The conditions for the extension theorem are the reason we are restricted
to the case of weakly geometric rough paths rather than general rough
paths. For the proof of this result in the general case we refer the
reader to \cite{LyonsVicExten}. For completeness we include a proof
in $\mathbb{R}^{d}$ for a generalized version following the same
argument as in \cite{LyonsVicExten} in the appendix. An interesting
point which distinguishes our setting from that of Malliavin calculus
is that for each perturbation of a path in a Cameron-Martin direction
there are infinitely many perturbations of the lifted rough path,
each corresponding to a choice of {}``cross-iterated integrals\textquotedblright{}
(projections) of the path and the Cameron-Martin direction. In other
words, for each Cameron-Martin direction, there are infinitely many
variations of the rough path which are not equivalent but have the
same variation at the path level.

We demonstrate that the tangent space is a well defined linear vector
space, and forms a bundle over the space of rough paths, but unfortunately
we do not believe that it is a fibre bundle. What is missing is the
structure of local trivialization. Yet, we are still able to solve
the flow equation

\[
C^{\prime}\left(\tau\right)=F\left(C\left(\tau\right)\right)\text{, }C\left(0\right)=X\]
for $\tau\in\left[0,T\right]$ on $WG\Omega_{p}\left(V\right)$ for
a class of functions $F$ which are Lipschitz in some sense to be
defined later (see definition \ref{def:Fcond}). Here we must consider
the derivative on the left as the tangent vector $\left[Z,\varphi\right]\left(\tau\right)$
uniquely associated to the curve $C\left(\tau\right)$ and $F$ as
assigning an element of the tangent space to each point of the curve.
Furthermore, we say a curve $U\left(\cdot\right)$ is a solution to
the flow equation if\[
\lim_{h\downarrow0}h^{-1}\left[d_{q}\left(U\left(\tau+h\right),V_{\left[F_{Z}\left(U\left(\tau\right)\right),F_{\varphi}\left(U\left(\tau\right)\right)\right]}\left(h\right)\right)\right]=0\]
for all $\tau$ and $U(0)=X$.

This definition requires some explanation. The reason for the appearance
of $d_{q}$ is that although for $p$-rough path space the natural
metric used in the solution definition should be $d_{p}$, due to
technical limitations we must use the metric $d_{q}$ for some $q>p$.
Specifically this is caused by the lack of an intrinsic compactness
theorem for sets in $\Omega_{p}$ which forces us to find compactness
in $\Omega_{q}$ for some $q>p$. Now, loosely speaking, since $V$
is a variational curve with parameter $h$ for the tangent vector
assigned to $U$ at time $\tau$, in a Banach space setting the above
would reduce to\[
\lim_{h\downarrow0}h^{-1}\left\Vert U\left(\tau+h\right)-\left[U\left(\tau\right)+hF\left(U\left(\tau\right)\right)\right]\right\Vert =0\text{,}\]
or equivalently,\[
\lim_{h\downarrow0}\left\Vert \frac{U\left(\tau+h\right)-U\left(\tau\right)}{h}-F\left(U\left(\tau\right)\right)\right\Vert =0\text{.}\]
Therefore our definition makes sense on the rough path space and seems
natural. However, the non-uniqueness of a variational curve $V$ associated
to a tangent vector means the definition may depend on the choice
of variational curve. We avoid this issue by only considering a class
of $F$ such that there exists a canonical choice of variational curve
associated to $F\left(U\left(\tau\right)\right)$ and construct a
solution via an Euler scheme.

Let us now more precisely state the main results we have briefly discussed
above.
\begin{thm*}
Let $C\left(\varepsilon\right):\left[-\tau,\tau\right]\rightarrow WG\Omega_{p}\left(V\right)$
be such that $C\left(0\right)=X$ and let \[
\left(1,C^{\prime}\left(0\right)^{1},C^{\prime}\left(0\right)^{2}\right)\]
be its derivative at $0$ in $T^{2}\left(V\right)$ sense. If $C^{\prime}\left(0\right)^{1}$
has finite $p$-variation and $C^{\prime}\left(0\right)^{2}$ has
finite $\frac{p}{2}$-variation, there exists a unique equivalence
class $\left[Z,\varphi\right]$ such that \[
\left.\frac{d}{d\varepsilon}\right\vert _{\varepsilon=0}C\left(\varepsilon\right)=\left.\frac{d}{d\varepsilon}\right\vert _{\varepsilon=0}V_{\left[Z,\varphi\right]}.\]
Where $\left[Z,\varphi\right]$ is a tangent vector and $V_{\left[Z,\varphi\right]}$$\left(\varepsilon\right)$
is any variational curve associated to it. 
\end{thm*}
That is, we identify the space of equivalence classes of curves which
abstractly define the tangent space. Furthermore, on such spaces we
have the following.
\begin{thm*}
If $F$ is a locally Lipschitz near $X_{0}$ vector field on $WG\Omega_{p}$,
then there exists a unique solution $U:\left[0,\alpha\right]\rightarrow\Omega_{q}\left(V\right)$
to the flow equation for $q>p$. 
\end{thm*}
Furthermore, if we strengthen the condition on $F$ to a globally
Lipschitz one then we obtain a global solution. This is the content
of theorem \ref{thm:GlobalSoln}. We refer the reader to definition
\ref{def:Fcond} for a precise explanation of what we mean by {}``$F$
is locally Lipschitz near $X_{0}$''.

There are a great number of papers dealing with tangent vectors of
Wiener space. Malliavin first introduced his idea of differentiating
a functional on Wiener space in \cite{MalDerivOrig}. Since then,
many topics have been developed and many powerful techniques have
been produced within his calculus. In particular, much work has been
done extending the Cameron-Martin quasi-invariance theorem to Wiener
space on a based manifold as a corollary to the existence of a flow
associated to some vector fields. Many of these papers consider more
general tangent vectors of Wiener space.

In \cite{Cruz1st}, Cruzeiro considered vector fields as maps from
Wiener space to the Cameron-Martin space, which satisfied certain
exponential integrability estimates. For these special vector fields,
the author was able to produce a solution flow by approximating the
fields through projection onto a finite dimensional space which was
identified with $\mathbb{R}^{n}$ for which ordinary differential
equation techniques produce a flow. As a corollary, it was obtained
that for any constant vector field (i.e. for all $x\in C\left(\left[0,1\right];R\right)$,
$F\left(x\right)=h$ for some $h$ in Cameron-Martin space) the measure
induced by the corresponding flow $U_{t}\left(x\right)$ is absolutely
continuous with respect to the Wiener measure. This paper instigated
a series of works related to extending the results to manifolds. In
\cite{MalliNatur,MalliLoops}, the authors showed that the Wiener
measure on loops over a compact connected Lie group is quasi-invariant
with respect to the left action of paths on the group which have finite
energy, i.e. $\int_{0}^{1}\left\Vert u^{-1}\left(t\right)\dot{u}\left(t\right)\right\Vert ^{2}dt<\infty$.
However, for the right action in \cite{MalliNatur}, Malliavin deals
with the Wiener measure on a connected Lie group of matrices, for
which a negative result is obtained. He defines the tangent space
as the space of continuous paths $u$ taking values in the Lie algebra
such that $\int_{0}^{1}\left\Vert \dot{u}\left(t\right)\right\Vert ^{2}dt<\infty$.
The main theorem then shows that if the adjoint representation is
not unitary, one cannot obtain quasi-invariance of the measure induced
by infinitesimal right action by elements of this tangent space. 

A major step in extension to manifolds was accomplished by Driver
in \cite{DriverFlow} where the author proves a quasi-invariance result
for the path space of a compact manifold without boundary, thereby
extending the work of Cruzeiro. He defines a tangent vector field
on the based path space $W$ as a map $X^{h}\left(\omega\right)\left(s\right):=H\left(\omega\right)\left(s\right)h\left(s\right)$
where $H\left(\omega\right)\left(s\right)$ is the stochastic parallel
translation along $\omega$ on the interval $\left[0,s\right]$, and
$h:\left[0,1\right]\rightarrow T_{o}M$ with $h\left(0\right)=o$
and $h$ has finite energy. The flow is constructed through geometric
means whenever the covarient derivative satisfies the torsion skew
symmetric condition. Furthermore, the quasi-invariance of the induced
measure is proved. Note that in the case that $M=\mathbb{R}^{n}$,
the tangent vector fields reduce to $X^{h}\left(\omega\right)\left(s\right)=h\left(s\right)$,
i.e. the usual Cameron-Martin space if we identify the tangent space
with $\mathbb{R}^{n}$. In this case, the flow is solved to be $u\left(t\right)=\omega+th$.
The torsion skew symmetric condition was relaxed in \cite{VFonPathSpace}
to allow for any affine connection which is adjoint skew symmetric
(if the affine connection preserves the Riemannian metric then the
adjoint skew symmetric and torsion skew symmetric conditions are equivalent).
An extension of these results is made in \cite{DriverFlow2} where
a flow is produced and a quasi-invariance theorem is shown for the
case of vector fields defined as above, but where $h$ is replaced
by a continuous semi-martingale of particular form.

A deterministic construction of Driver's flow on a closed Riemannian
manifold is produced in \cite{FlowOnRP} using the theory of rough
paths. Here Lyons and Qian construct a flow for a class of vector
fields obtained from solving a class of rough differential equations.
They apply this to the construction of Driver's flow using the fact
that one can solve the flow equation for a geometric vector field
by considering the solution flow of an Ito map obtained by solving
a certain differential equation.

A further notion of tangent spaces was investigated by Cipriano, Cruziro,
and Malliavin in \cite{TgtProcessFlow,TgtProcessDef}. In these works,
the authors develop the notion of a process tangent to Wiener space
defined to be an $\mathbb{R}^{d}$ valued semi-martingale $\xi$ satisfying
the It\^{o} equation $d\xi_{i}\left(t\right)=A_{ij}dx_{j}\left(t\right)+B_{j}dt$
where the anti-symmetric matrix coefficients $A_{ij}$ are semi-martingales
which also have a Stratonovich representation. An associated flow
is also constructed and shown to have a push forward measure which
is absolutely continuous with respect to the Wiener measure with density
in $L^{p}\left(d\mu\right)$ for all $p\geq1$. Of particular interest
is the fact that a tangent process has a representation as the solution
to the equation $\eta_{t}=\int_{0}^{t}\eta_{t}dx\left(t\right)+\gamma dt$.
In some sense, this is related to the information contained in cross
iterated integrals of a process with the Brownian path together with
some additive function. This information is contained in our definition
of the tangent in a deterministic way. 

Yet another approach to defining derivatives of functionals on the
Wiener space is considered in \cite{RotationVar}. Instead of considering
the variation $F\left(x+th\right)$, the authors consider a class
of measure preserving transformations $T_{t}$ giving associated variation
$F\left(T_{t}x+th\right)$ and construct a solution to the related
flow equation. 

We stress that in the above non-rough path approaches, the vector
fields are limited to a subspace of the Wiener space (usually the
Cameron-Martin space) and the flow is defined for almost every (with
respect to Wiener measure) element of continuous path space. In our
definition we provide a much larger class of tangent directions which
is defined point wise without reference to a measure. In fact, given
any element of the same rough path space, there are infinitely many
tangent directions which in some sense are variations in the direction
of the given rough path.

The paper is organized as follows. In section 2 we formally present
the machinery we require from rough path theory with the exception
of the Lyons-Victoir extension theorem which is instead discussed
in the appendix. Section 3 is devoted to the definition and properties
of the tangent space while the final section 4 covers the construction
of the local and global flows.

\section{Preliminaries}

Let us discuss the tools we will need from rough path theory. Historically
the theory of rough paths was developed as an approach to making deterministic
sense of differential equations of the type\begin{align*}
dx_{t}= & F\left(x_{t}\right)dy_{t}\\
x_{0}= & \xi\end{align*}
where the path $y$ is very irregular in time parameter $t$. In the
case that $y$ is not differentiable, solutions $x$ must interpreted
as an integral \begin{align*}
x_{\cdot}= & x_{0}+\int_{0}^{\cdot}F\left(x_{t}\right)dy_{t}\end{align*}
and if $y$ is regular enough to have bounded variation, then ODE
theory tells us that, for Lipschitz $F$, the above equation has a
bounded variation solution. Rough path theory provides an extension
of the classical theory significantly beyond bounded variation (Young
integration can be seen as a simpler extension). In order to go farther,
more information than the path increments is required to define the
integral. And this was conjectured in some sense by F\"{o}llmer (see
the history section of Lyon's original paper \cite{LyonsIbero}).
More specifically, he guessed that knowing the increment and Levy
area would be sufficient to solve stochastic differential equations.
In fact, Lyons showed that a deterministic approach using just the
increments is not possible. With this in mind and the fact that higher
order iterated integrals contain area information, we want to be able
to define the iterated integrals \begin{align*}
\int_{s<t_{1}<\cdots<t_{n}<t}dY_{t_{1}}\otimes\cdots\otimes dY_{t_{n}}\end{align*}
for paths of finite $p$-variation.

Notice that if $y$ has finite variation, then the first iterated
integral is just the increment $y_{t}-y_{s}$ and the second iterated
integral is $\lim_{m\left(\mathcal{D}\right)\rightarrow0}\sum_{l}\left(y_{t_{l}}-y_{s}\right)\otimes\left(y_{t}-y_{t_{l}}\right)$
so in this case, the higher order term is determined by the increment.
Also, since up to this point we don't have an integration theory for
rough paths, such objects are not defined for $p$-variation paths
for $p\geq2$. Instead, we consider a $p$-rough path as an object
in $\oplus_{i=i}^{\left\lfloor p\right\rfloor }V^{\otimes i}$ where
the element in $V^{\otimes i}$ behaves algebraically like an $i$th
order iterated integral and satisfies a finite $p$-variation condition.
The reason we consider elements only up to order $\left\lfloor p\right\rfloor $
is that for any $p$-rough path, there exists a unique extension to
$\oplus_{i=i}^{\infty}V^{\otimes i}$ which has finite $p$-variation
(see for example the extension theorem 3.1.2 in \cite{SysControlRP}).

Let us now give the formal framework for considering path increments
as well as higher order elements.
\begin{defn}
Let $V$ be a Banach space and $T^{n}\left(V\right)=\oplus_{i=0}^{n}V^{\otimes i}$
be its tensor powers. A continuous function $X:\Delta_{T}\rightarrow T^{n}\left(V\right)$
denoted at each pair $\left(s,t\right)$ by \begin{align*}
X_{s,t}= & \left(X_{s,t}^{0},\ldots,X_{s,t}^{n}\right)\end{align*}
is said to be a multiplicative functional of degree $n$ in $V$ if
\begin{align}
X_{s,t}\otimes X_{t,u}= & X_{s,u}\label{eq:Chen}\end{align}
for all $s,t,u\in\left[0,T\right]$ satisfying $s<t<u$, where the
tensor product is taken in $T^{n}\left(V\right)$.
\end{defn}
The algebraic condition (\ref{eq:Chen}), referred to throughout as
Chen's identity, captures the additivity property of integrals over
regions. Indeed, for $x$, a path with bounded variation, letting
$X_{s,u}^{2}=\int_{s<t_{1}<t_{2}<u}dx_{t_{1}}\otimes dx_{t_{2}}$,
we have \begin{align*}
\int_{s<t_{1}<t_{2}<u}dx_{t_{1}}\otimes dx_{t_{2}}= & \int_{s}^{t}\left(x_{t_{2}}-x_{s}\right)\otimes dx_{t_{2}}+\int_{t}^{u}\left(x_{t_{2}}-x_{t}\right)\otimes dx_{t_{2}}\\
 & +\int_{t}^{u}\left(x_{t}-x_{s}\right)\otimes dx_{t_{2}}\\
= & \int_{s<t_{1}<t_{2}<t}dx_{t_{1}}\otimes dx_{t_{2}}+\int_{t<t_{1}<t_{2}<u}dx_{t_{1}}\otimes dx_{t_{2}}\\
 & +\left(x_{t}-x_{s}\right)\otimes\left(x_{u}-x_{t}\right)\\
= & X_{s,t}^{2}+X_{t,u}^{2}+X_{s,t}^{1}\otimes X_{t,u}^{1}\\
= & X_{s,t}\otimes X_{t,u}.\end{align*}
Finally, we form a $p$-rough path by the imposition of the analytic
finite $p$-variation condition on $X$. 
\begin{defn}
A $p$-rough path in $V$ is a multiplicative functional of degree
$\left\lfloor p\right\rfloor $ in $V$ with finite $p$-variation,
i.e.\begin{align*}
\sup_{\mathcal{D}\subseteq\left[0,T\right]}\sum_{\mathcal{D}}\left\Vert X_{t_{l},t_{l+1}}^{i}\right\Vert _{V}^{\frac{p}{i}}< & \infty\end{align*}
for each $i\in\left\{ 1,\ldots,\left\lfloor p\right\rfloor \right\} $.
The space of all $p$-rough paths in $V$ is denoted by $\Omega_{p}\left(V\right)$
and can be equipped with the distance \begin{align*}
d_{p}\left(X,Y\right):= & \max_{i\in\left\{ 1,\ldots,\left\lfloor p\right\rfloor \right\} }\sup_{\mathcal{D}\subseteq\left[0,T\right]}\left(\sum_{\mathcal{D}}\left\Vert X_{t_{l},t_{l+1}}^{i}\right\Vert _{V}^{\frac{p}{i}}\right)^{\frac{i}{p}}\end{align*}
in which case it is a complete metric space. 
\end{defn}
There are two special spaces of rough paths. The $p$-rough paths
which are the limit, in $p$-variation distance, of a sequence of
smooth rough paths are called geometric $p$-rough paths and denoted
$G\Omega_{p}\left(V\right)$. And the elements of $\Omega_{p}\left(V\right)$
that can be realized as the limit in uniform topology of the canonical
lifts of bounded $p$-variation smooth paths are called weakly geometric
$p$-rough paths and are denoted $WG\Omega_{p}\left(V\right)$. The
strict inclusions\[
G\Omega_{p}\left(V\right)\subset WG\Omega_{p}\left(V\right)\subset\Omega_{p}\left(V\right)\]
hold. 

The following is standard and shows that up to reparameterization,
$p$-rough paths are closely related to $\frac{1}{p}$-H\"{o}lder
continuous paths.
\begin{prop}
\label{pro:HolderCont}Let $X\in\Omega_{p}\left(V\right)$ and assume
$X^{1}$ is not zero on any interval. Also define $\tau:\left[0,T\right]\rightarrow\mathbb{R}^{+}$
by \[
\tau\left(t\right)=\frac{\omega\left(t\right)T}{\omega\left(T\right)}\]
 where $\omega\left(t\right)$ is the $p$th power of the $p$-variation
of $X$ up to time $t$ on the path level, that is\[
\omega\left(t\right)=\sum_{i=1}^{2}\sup_{\mathcal{D}\subseteq\left[0,t\right]}\sum_{\mathcal{D}}\left|X_{t_{l-1}t_{l}}^{i}\right|^{\frac{p}{i}}.\]
Then\begin{align*}
\left|X_{\tau^{-1}\left(s\right),\tau^{-1}\left(t\right)}^{i}\right| & \leq\left(\frac{\omega\left(T\right)}{T}\right)^{\frac{i}{p}}\left(t-s\right)^{\frac{i}{p}}.\end{align*}
\end{prop}
\begin{proof}
We have by the sub-additivity of $p$-variation over subintervals,
\begin{align*}
\left|X_{\tau^{-1}\left(s\right)\tau^{-1}\left(t\right)}^{i}\right|^{\frac{p}{i}} & \leq\sum_{i=1}^{2}\sup_{\mathcal{D}\subseteq\left[\tau^{-1}\left(s\right),\tau^{-1}\left(t\right)\right]}\sum_{\mathcal{D}}\left|X_{t_{l-1}t_{l}}^{i}\right|^{\frac{p}{i}}\\
 & \leq\sum_{i=1}^{2}\sup_{\mathcal{D}\subseteq\left[0,\tau^{-1}\left(t\right)\right]}\sum_{\mathcal{D}}\left|X_{t_{l-1}t_{l}}^{i}\right|^{\frac{p}{i}}-\sum_{i=1}^{2}\sup_{\mathcal{D}\subseteq\left[0,\tau^{-1}\left(s\right)\right]}\sum_{\mathcal{D}}\left|X_{t_{l-1}t_{l}}^{i}\right|^{\frac{p}{i}}\\
 & =\omega\left(\tau^{-1}\left(t\right)\right)-\omega\left(\tau^{-1}\left(s\right)\right).\end{align*}
Then by the definition of $\tau$, \begin{align*}
\omega\left(t\right)= & \frac{\omega\left(T\right)}{T}\tau\left(t\right)\end{align*}
so plugging in $\tau^{-1}$, \begin{align*}
\omega\left(\tau^{-1}\left(t\right)\right)= & \frac{\omega\left(T\right)}{T}t\end{align*}
and hence \begin{align*}
\left|X_{\tau^{-1}\left(s\right)\tau^{-1}\left(t\right)}^{i}\right|^{\frac{p}{i}}\leq & \frac{\omega\left(T\right)}{T}\left(t-s\right).\end{align*}
Finally, taking the $p$th root gives the result.
\end{proof}
The next bound will be used to prove the well known extrinsic compactness
result for sets in $\Omega_{p}\left(V\right)$. 
\begin{lem}
Let $X,Y\in\Omega_{p}\left(V\right)$. Then for $q>p$ we have the
following bound \begin{equation}
d_{q}\left(X,Y\right)\leq Cd_{p}\left(X,Y\right)^{\frac{p}{q}}.\label{eq:PVarBoundsqVar}\end{equation}
where \begin{align*}
C= & \max\left\{ \left(2\sup_{t}\left|X_{0,t}^{1}-Y_{0,t}^{1}\right|\right)^{\frac{q-p}{q}},\right.\\
 & \left.\left(2\sup_{t}\left|X_{0,t}^{2}-Y_{0,t}^{2}\right|\left(1+2\left(\sup_{t}\left|X_{0,t}^{1}\right|+\sup_{t}\left|Y_{0,t}^{1}\right|\right)\right)\right)^{\frac{q-p}{q}}\right\} \end{align*}
\end{lem}
\begin{proof}
First we make the elementary estimate \begin{align*}
 & \sup_{\mathcal{D}}\left(\sum_{l}\left|X_{t_{l}t_{l+1}}^{i}-Y_{t_{l}t_{l+1}}^{i}\right|^{\frac{q}{i}}\right)^{\frac{i}{q}}\\
 & =\sup_{\mathcal{D}}\left(\sum_{l}\left|X_{t_{l}t_{l+1}}^{i}-Y_{t_{l}t_{l+1}}^{i}\right|^{\frac{q-p}{i}}\left|X_{t_{l}t_{l+1}}^{i}-Y_{t_{l}t_{l+1}}^{i}\right|^{\frac{p}{i}}\right)^{\frac{i}{q}}\\
 & \leq\left(\sup_{s,t\in\Delta_{T}}\left|X_{s,t}^{i}-Y_{s,t}^{i}\right|\right)^{\frac{q-p}{q}}\sup_{\mathcal{D}}\left(\sum_{l}\left|X_{t_{l}t_{l+1}}^{i}-Y_{t_{l}t_{l+1}}^{i}\right|^{\frac{p}{i}}\right)^{\frac{i}{q}}.\end{align*}
Then, by Chen's identity,\begin{align*}
\left|X_{s,t}^{1}-Y_{s,t}^{1}\right| & =\left|X_{0,t}^{1}-X_{0,s}^{1}-\left(Y_{0,t}^{1}-Y_{0s}^{1}\right)\right|\\
 & \leq2\sup_{t}\left|X_{0,t}^{1}-Y_{0,t}^{1}\right|\end{align*}
and \begin{align*}
\left|X_{s,t}^{2}-Y_{s,t}^{2}\right| & =\left|X_{0,t}^{2}-X_{0,s}^{2}-X_{0,s}^{1}\otimes X_{s,t}^{1}-\left(Y_{0,t}^{2}-Y_{0,s}^{2}-Y_{0,s}^{1}\otimes Y_{s,t}^{1}\right)\right|\\
 & \leq2\sup_{t}\left|X_{0,t}^{2}-Y_{0,t}^{2}\right|+\left|\left(X_{0,s}^{1}-Y_{0,s}^{1}\right)\otimes X_{s,t}^{1}+Y_{0,s}^{1}\otimes\left(X_{s,t}^{1}-Y_{s,t}^{1}\right)\right|\\
 & \leq2\sup_{t}\left|X_{0,t}^{2}-Y_{0,t}^{2}\right|+4\sup_{t}\left|X_{0,t}^{1}-Y_{0,t}^{1}\right|\left(\sup_{t}\left|X_{0,t}^{1}\right|+\sup_{t}\left|Y_{0,t}^{1}\right|\right)\end{align*}
which gives the result.
\end{proof}
We will use the following compactness result in the proof of the existence
of flow equation solutions on the space of rough paths. The fact that
it is not intrinsic is the source of the awkward fact that our solution
of a flow equation on $WG\Omega_{p}$ lives in $\Omega_{q}$.
\begin{thm}
\label{thm:ExtrinCompac}If $V$ is finite dimensional then any $\mathcal{A}\subset\Omega_{p}\left(V\right)$
satisfying \[
\sup_{\mathcal{A}}d_{p}\left(0,X\right)\leq M\]
is relatively compact.\end{thm}
\begin{proof}
Consider the family of paths in $V$ given by $\left\{ X_{0,\cdot}^{1}:X\in\mathcal{A}\right\} $.
The uniformly bounded $p$-variation implies\[
\sup_{\mathcal{A}}\left|X_{s,t}^{1}\right|\leq M\]
so that because $V$ is finite dimensional, for fixed $t$ $\left\{ X_{0,t}^{1}:X\in\mathcal{A}\right\} $
is relatively compact. Also, by proposition \ref{pro:HolderCont}
\begin{align*}
\left|X_{0,t}^{1}-X_{0,s}^{1}\right| & =\left|X_{s,t}^{1}\right|\\
 & \leq C\left|t-s\right|^{\frac{1}{p}}\end{align*}
where the constant depends only on the length of the time interval
and the total $p$-variation of $X$. So we have shown equicontinuity.
Hence, the Ascoli-Arzela theorem implies that $\left\{ X_{0,\cdot}^{1}:X\in\mathcal{A}\right\} $
is relatively compact in the uniform topology on $C\left(\left[0,T\right],V\right)$. 

Now, let $X\left(n\right)$ be any sequence in $\mathcal{A}$. Then
by the above argument there exists a subsequence $X\left(n_{k}\right)$
such that $X_{0,\cdot}^{1}\left(n_{k}\right)$ converges in the uniform
norm. Next consider the family of paths in $V^{\otimes2}$ given by
$\left\{ X_{0,\cdot}^{2}\left(n_{k}\right)\right\} $. Again we have
pointwise relative compactness in $V^{\otimes2}$ due to finite dimensionality
and uniformly bounded $p$-variation. Also, just as above, we apply
Chen's identity and proposition \ref{pro:HolderCont} to arrive at
\begin{align*}
\left|X_{0,t}^{2}\left(n_{k}\right)-X_{0,s}^{2}\left(n_{k}\right)\right| & =\left|X_{s,t}^{2}\left(n_{k}\right)+X_{0,s}^{1}\left(n_{k}\right)\otimes X_{s,t}^{1}\left(n_{k}\right)\right|\\
 & \leq\left|X_{s,t}^{2}\left(n_{k}\right)\right|+\left|X_{0,s}^{1}\left(n_{k}\right)\right|\left|X_{s,t}^{1}\left(n_{k}\right)\right|\\
 & \leq C\left[\left(t-s\right)^{\frac{2}{p}}+s^{\frac{1}{p}}\left(t-s\right)^{\frac{1}{p}}\right]\end{align*}
for $C$ independent of $k$. Then we have equicontinuity of the paths
$X_{0,\cdot}^{2}$ as \[
\left|t-s\right|\leq\min\left\{ \left(\frac{\varepsilon}{2C}\right)^{\frac{p}{2}},\left(\frac{\varepsilon}{2CT^{\frac{1}{p}}}\right)^{p}\right\} \]
implies $\left|X_{0,t}^{2}\left(n_{k}\right)-X_{0,s}^{2}\left(n_{k}\right)\right|\leq\varepsilon$.
Therefore, after a second application of the Ascoli-Arzela theorem
we can extract a further subsequence such that both $X_{0,\cdot}^{1}\left(n_{k_{l}}\right)$
and $X_{0,\cdot}^{2}\left(n_{k_{l}}\right)$ converge in the uniform
norm on $V$ and $V^{\otimes2}$ respectively.

Finally, consider any sequence in $\mathcal{A}$. The above allows
us to extract a subsequence such that each $X_{0,\cdot}^{i}\left(l\right)$
converges in uniform norm $V^{\otimes i}.$ Then the bound (\ref{eq:PVarBoundsqVar})
together with the uniformly bounded $p$-variation implies the subsequence
is also Cauchy in $\Omega_{q}\left(V\right)$. Therefore completeness
of $\Omega_{q}\left(V\right)$ gives the result.
\end{proof}
Let us now introduce some notation and basic operations on rough paths.
We will often consider rough paths which at the first tensor level
consist of a pair of rough paths, i.e. rough paths in $V\oplus W$,
so we use the following notation for the component parts of such rough
paths $Z$. We use \begin{align*}
Z^{1}= & \left(\pi_{1}\left(Z\right)^{1},\pi_{2}\left(Z\right)^{1}\right)\\
Z^{2}= & \left(\begin{array}{cc}
\pi_{1}\left(Z\right)^{2} & \pi_{1,2}\left(Z\right)\\
\pi_{2,1}\left(Z\right) & \pi_{2}\left(Z\right)^{2}\end{array}\right)\end{align*}
due to the decomposition\[
T^{2}\left(V\oplus W\right)=1\oplus\left(V\oplus W\right)\oplus\left(V^{\otimes2}\oplus\left(V\otimes W\right)\oplus\left(W\otimes V\right)\oplus W^{\otimes2}\right),\]
and may also denote $Z$ by $\left(\pi_{1}\left(Z\right),\pi_{2}\left(Z\right)\right)$where
the projections \begin{align*}
\pi_{1}\left(Z\right)= & \left(1,\pi_{1}\left(Z\right)^{1},\pi_{1}\left(Z\right)^{2}\right)\\
\pi_{2}\left(Z\right)= & \left(1,\pi_{2}\left(Z\right)^{1},\pi_{2}\left(Z\right)^{2}\right)\end{align*}
are $p$-rough paths in $V$ and $W$ respectively. In this case,
by identifying $\oplus$ with $+$ when the objects are in the same
space, Chen's identity (\ref{eq:Chen}) is equivalent to \begin{align*}
\pi_{i}\left(Z\right)_{s,u}^{1}= & \pi_{i}\left(Z\right)_{s,t}^{1}+\pi_{i}\left(Z\right)_{t,u}^{1}\\
\pi_{i}\left(Z\right)_{s,u}^{2}= & \pi_{i}\left(Z\right)_{s,t}^{2}+\pi_{i}\left(Z\right)_{t,u}^{2}+\pi_{i}\left(Z\right)_{s,t}^{1}\otimes\pi_{i}\left(Z\right)_{t,u}^{1}\\
\pi_{ij}\left(Z\right))_{s,u}= & \pi_{ij}\left(Z\right))_{s,t}+\pi_{ij}\left(Z\right))_{t,u}+\pi_{i}\left(Z\right)_{s,t}^{1}\otimes\pi_{j}\left(Z\right)_{t,u}^{1}.\end{align*}
If we know two rough paths $X$ and $Y$ as a single rough path $Z$
in $V\oplus V$, i.e. in some sense we know their cross iterated integrals,
then we may add the projections of $Z$ in the following sense. 
\begin{prop}
Let $Z=\left(X,Y\right)$ be a $p$-rough path in $V\oplus V$. Then
\[
\left(1,X^{1}+Y^{1},X^{2}+\pi_{1,2}\left(Z\right)+\pi_{2,1}\left(Z\right)+Y^{2}\right)\]
is a $p$-rough path in $V$. \end{prop}
\begin{proof}
This is from Chen's identity for rough paths in $V\oplus V$ given
above in equation (\ref{eq:Chen}).
\end{proof}
Scalar multiplication is well defined for rough paths in the following
sense.
\begin{prop}
Let $X$ be a $p$-rough path in $V$ and for all $\lambda\in\mathbb{R}$
define $\lambda X$ by \begin{align*}
\left(\lambda X\right)^{1}= & \lambda X^{1}\\
\left(\lambda X\right)^{2}= & \lambda^{2}X^{2}.\end{align*}
 Then, $\lambda X$ is a $p$-rough path over $V$.\end{prop}
\begin{proof}
Immediate from Chen's identity and the properties of scalar multiplication
on tensor product spaces.
\end{proof}

\section{Tangent Space Construction}

For a functional $f:\Omega_{p}\left(\mathbb{R}^{d}\right)\rightarrow\mathbb{R}$
one might try to define its derivative by considering $\left.\frac{d}{d\varepsilon}\right|_{\varepsilon=0}f\left(X+\varepsilon Y\right)$.
If we could define this object, it would give us the derivative of
$f$ at $X$ in the direction $Y$. However, as rough path space is
not linear, it is not clear how we should vary $X$ as addition of
rough paths does not make sense. Even with this fact, we can define
a kind of addition of $X$ and $Y$ if we have some information not
contained in just $X$ and $Y$. That is, if we know $Z\in\Omega_{p}\left(V\oplus V\right)$
with $\pi_{1}\left(Z\right)=X,$$\pi_{2}\left(Z\right)=Y$. Naturally,
there are many such $Z$ that give rise to $X$ and $Y$ through the
projections, but they give different {}``sums'' as the definition
of the {}``sum'' depends on cross iterated integrals or the projections
$\pi_{1,2}\left(Z\right)$ and $\pi_{2,1}\left(Z\right)$. This tells
us that the information in $Y$ alone is not sufficient to determine
a direction in rough path sense. Yet $Y$ together with the cross
iterated integrals of $X$ and $Y$ is enough to determine the sum
of $X$ and $Y$ in rough path sense. This suggests as a first step
we consider elements $Z$ of $\Omega_{p}\left(V\mathbb{\oplus}V\right)$
such that $\pi_{1}\left(Z\right)=X$. For such $Z$ we can {}``add''
$\pi_{1}\left(Z\right)$ and $\pi_{2}\left(Z\right)$ and form a new
rough path $K$ in the same space as each of the projections according
to the formula \begin{align}
K_{s,t}^{1}:= & \pi_{1}\left(Z\right)_{s,t}^{1}+\pi_{2}\left(Z\right)_{s,t}^{1}\label{eq:RPAdd1}\\
K_{s,t}^{2}:= & \pi_{1}\left(Z\right)_{s,t}^{2}+\pi_{2}\left(Z\right)_{s,t}^{2}+\pi_{12}\left(Z\right)_{s,t}+\pi_{21}\left(Z\right)_{s,t}\label{eq:RPAdd2}\end{align}
where \begin{align*}
Z_{s,t}= & \left(1,Z_{s,t}^{1},Z_{s,t}^{2}\right)\\
= & \left(1,\pi_{1}\left(Z\right)_{s,t}^{1},\pi_{2}\left(Z\right)_{s,t}^{1},\pi_{1}\left(Z\right)_{s,t}^{2},\pi_{12}\left(Z\right)_{s,t},\pi_{21}\left(Z\right)_{s,t},\pi_{2}\left(Z\right)_{s,t}^{2}\right).\end{align*}
Then, \begin{align*}
K_{s,t}^{1}\left(\varepsilon\right):= & \pi_{1}\left(Z\right)_{s,t}^{1}+\varepsilon\pi_{2}\left(Z\right)_{s,t}^{1}\\
K_{s,t}^{2}\left(\varepsilon\right):= & \pi_{1}\left(Z\right)_{s,t}^{2}+\varepsilon\left(\pi_{12}\left(Z\right)_{s,t}+\pi_{21}\left(Z\right)_{s,t}\right)+\varepsilon^{2}\pi_{2}\left(Z\right)_{s,t}^{2}\end{align*}
is the rough paths equivalent of adding $\varepsilon$ times a direction
to $X$. 
\begin{prop}
\label{pro:RPAdd}If $Z\in\Omega_{p}\left(V\oplus V\right)$ is of
the form\begin{align*}
Z_{s,t}^{1}= & \left(\pi_{1}\left(Z\right)^{1},\pi_{2}\left(Z\right)^{1}\right)\\
Z_{s,t}^{2}= & \left(\begin{array}{cc}
\pi_{1}\left(Z\right)_{s,t}^{2} & \pi_{12}\left(Z\right)_{s,t}\\
\pi_{21}\left(Z\right)_{s,t} & \pi_{2}\left(Z\right)_{s,t}^{2}\end{array}\right)\end{align*}
then the path $V_{Z}\left(\varepsilon\right)$ defined by\begin{align*}
V_{Z}\left(\varepsilon\right)_{s,t}^{1}= & \left(\pi_{1}\left(Z\right)_{s,t}^{1},\varepsilon\pi_{2}\left(Z\right)_{s,t}^{1}\right)\\
V_{Z}\left(\varepsilon\right)_{s,t}^{2}= & \left(\begin{array}{cc}
\pi_{1}\left(Z\right)_{s,t}^{2} & \varepsilon\pi_{12}\left(Z\right)_{s,t}\\
\varepsilon\pi_{21}\left(Z\right)_{s,t} & \varepsilon^{2}\pi_{2}\left(Z\right)_{s,t}^{2}\end{array}\right)\end{align*}
is also in $\Omega_{p}\left(V\oplus V\right)$. In particular, this
implies the function $V\left(\varepsilon\right)$ defined by \begin{align*}
V\left(\varepsilon\right)_{s,t}^{1}= & \pi_{1}\left(Z\right)_{s,t}^{1}+\varepsilon\pi_{2}\left(Z\right)_{s,t}^{1}\\
V\left(\varepsilon\right)_{s,t}^{2}= & \pi_{1}\left(Z\right)_{s,t}^{2}+\varepsilon\pi_{12}\left(Z\right)_{s,t}+\varepsilon\pi_{21}\left(Z\right)_{s,t}+\varepsilon^{2}\pi_{2}\left(Z\right)_{s,t}^{2}\end{align*}
is in $\Omega_{p}\left(V\right)$.\end{prop}
\begin{proof}
We must verify that Chen's identify is satisfied so we examine $V_{Z}\left(\varepsilon\right)_{s,t}\otimes V_{Z}\left(\varepsilon\right)_{tu}$.
For simplicity let $X$ and $Y$ denote $\pi_{1}\left(Z\right)$ and
$\pi_{2}\left(Z\right)$ respectively. By definition, \begin{align*}
 & \left(V_{Z}\left(\varepsilon\right)_{s,t}\otimes V_{Z}\left(\varepsilon\right)_{t,u}\right)^{2}=\\
 & \left(\begin{array}{cc}
X_{s,t}^{2}+X_{t,u}^{2}+X_{s,t}^{1}\otimes X_{t,u}^{1} & \varepsilon\left(\pi_{1,2}\left(Z\right)_{s,t}+\pi_{1,2}\left(Z\right)_{t,u}+X_{s,t}^{1}\otimes Y_{t,u}^{1}\right)\\
\varepsilon\left(\pi_{2,1}\left(Z\right)_{s,t}+\pi_{2,1}\left(Z\right)_{t,u}+Y_{s,t}^{1}\otimes X_{t,u}^{1}\right) & \varepsilon^{2}\left(Y_{s,t}^{2}+Y_{t,u}^{2}+Y_{s,t}^{1}\otimes Y_{t,u}^{1}\right)\end{array}\right)\end{align*}
 But since $Z$ is multiplicative, we have \begin{align*}
 & \left(Z_{s,t}\otimes Z_{t,u}\right)^{2}=\\
 & \left(\begin{array}{cc}
X_{s,t}^{2}+X_{t,u}^{2}+X_{s,t}^{1}\otimes X_{t,u}^{1} & \pi_{1,2}\left(Z\right)_{s,t}+\pi_{1,2}\left(Z\right)_{t,u}+X_{s,t}^{1}\otimes Y_{t,u}^{1}\\
\pi_{2,1}\left(Z\right)_{s,t}+\pi_{2,1}\left(Z\right)_{t,u}+Y_{s,t}^{1}\otimes X_{t,u}^{1} & Y_{s,t}^{2}+Y_{t,u}^{2}+Y_{s,t}^{1}\otimes Y_{t,u}^{1}\end{array}\right)\\
 & =\left(\begin{array}{cc}
X_{s,u}^{2} & \pi_{12}\left(Z\right)_{s,u}\\
\pi_{21}\left(Z\right)_{s,u} & Y_{s,u}^{2}\end{array}\right)\end{align*}
so that \begin{align*}
X_{s,t}^{2}+X_{t,u}^{2}+X_{s,t}^{1}\otimes X_{t,u}^{1}= & X_{s,u}^{2}\\
\pi_{1,2}\left(Z\right)_{s,t}+\pi_{1,2}\left(Z\right)_{t,u}+X_{s,t}^{1}\otimes Y_{t,u}^{1}= & \pi_{1,2}\left(Z\right)_{s,u}\\
\pi_{2,1}\left(Z\right)_{s,t}+\pi_{2,1}\left(Z\right)_{t,u}+Y_{s,t}^{1}\otimes X_{t,u}^{1}= & \pi_{1,2}\left(Z\right)_{s,u}\\
Y_{s,t}^{2}+Y_{t,u}^{2}+Y_{s,t}^{1}\otimes Y_{t,u}^{1}= & Y_{s,u}^{2}\end{align*}
which implies\begin{align*}
\left(V_{Z}\left(\varepsilon\right)_{s,t}\otimes V_{Z}\left(\varepsilon\right)_{t,u}\right)^{2}= & V_{Z}\left(\varepsilon\right)_{s,u}^{2}.\end{align*}
\end{proof}
\begin{rem}
The path $\varepsilon\rightarrow V_{Z}\left(\varepsilon\right)$ defined
above is continuous in $p$ variation topology. 
\end{rem}
If $c\left(\varepsilon\right)$ is a curve consisting only of smooth
rough paths such that $c\left(0\right)=X$, then the information in
the derivative at $0$ is $c'\left(0\right)^{1}$, together with $c'\left(0\right)^{2}$
which corresponds to the sum $\int dX\otimes dc'\left(0\right)^{1}+\int c'\left(0\right)d\otimes X$
. If we want to associate it to some equivalence class $\left[Z\right]$
for some $Z\in WG\Omega_{p}\left(V\oplus V\right)$, we would need
a way of separating $c'\left(0\right)^{2}$ into two terms which would
correspond to $\pi_{1,2}\left(Z\right)$ and $\pi_{2,1}\left(Z\right)$.
If we could achieve this, then to any equivalence class of curves
we could uniquely associate a $\left[Z\right]$ where $Z^{1}=\left(X,c'\left(0\right)^{1}\right)$
and $\left.\frac{d}{d\varepsilon}\right|_{\varepsilon=0}V_{Z}\left(\varepsilon\right)=\left.\frac{d}{d\varepsilon}\right|_{\varepsilon=0}c\left(\varepsilon\right)$.
Unfortunately, there is no clear way to do this other than for smooth
rough paths. Hence, we need something further which is given by the
independent second level variation. 

Suppose $X=\left(1,X^{1},X^{2}\right)$ is a $p$-geometric rough
path over $V$ and $c\left(\varepsilon\right)$ is a curve in $G\Omega_{p}\left(\mathbb{R}^{d}\right)$
such that $\frac{d}{d\varepsilon}c\left(\varepsilon\right)^{1}$ has
finite $p$ variation and $c\left(0\right)=X$. We can form the Lyons-Victoir
extension $W\in WG\Omega_{p}\left(V\oplus V\right)$ of $W^{1}=\left(X^{1},c'\left(0\right)^{1}\right)$.
Intuitively, the projections $\pi_{1,2}\left(W\right)$ and $\pi_{2,1}\left(W\right)$
are some cross iterated integrals of $X$ and $c'\left(0\right)^{1}$.
But, the sum of the projections $\pi_{1,2}$ and $\pi_{2,1}$ may
not equal $c'\left(0\right)^{2}$ due to independent second level
variation. Therefore, we introduce $\varphi=c'\left(0\right)^{2}-\pi_{1,2}\left(W\right)-\pi_{2,1}\left(W\right)$
which measures the relationship of the extension of $\left(X^{1},c'\left(0\right)^{1}\right)$
to the derivative. Then, the pair $\left(W,\varphi\right)$ contains
all the information included in $X^{1}$ and the derivative of $c\left(\varepsilon\right)$.
In order to recover the information from $X^{2}$, we form a new element
$Z$ of $WG\Omega_{p}\left(V\oplus V\right)$ by replacing $\pi_{1}\left(W\right)^{2}$
with $X^{2}$. Then $\left(Z,\varphi\right)$ has all the information
contained in the derivative of $c\left(\varepsilon\right)$, though
we need to quotient out some information because $\pi_{2}\left(Z\right)^{2}$
is not related to the derivative of $c\left(\varepsilon\right)$. 

We now formalize the above considerations. 
\begin{defn}
\label{def:TgtEquiv}For $Z,\tilde{Z}\in WG\Omega_{p}\left(V\oplus V\right)$
such that $\pi_{1}\left(Z\right)=\pi_{1}\left(\tilde{Z}\right)=X$,
and $\varphi,\tilde{\varphi}\in\Omega_{\frac{p}{2}}\left(V\oplus V\right)$,
we say the pair $\left(Z,\varphi\right)$ is equivalent to $\left(\tilde{Z},\tilde{\varphi}\right)$,
denoted $\left(Z,\varphi\right)\sim\left(\tilde{Z},\tilde{\varphi}\right)$,
if \begin{align*}
\left.\frac{d}{d\varepsilon}\right|_{\varepsilon=0}V_{\left(Z,\varphi\right)}= & \left.\frac{d}{d\varepsilon}\right|_{\varepsilon=0}V_{\left(\tilde{Z},\tilde{\varphi}\right)}\end{align*}
where the variational curve $V_{\left(Z,\varphi\right)}^{1}\left(\varepsilon\right)\in C_{0}\left(\Delta_{T};T^{2}\left(V\right)\right)$
is defined by \begin{align*}
V_{\left(Z,\varphi\right)}^{1}\left(\varepsilon\right)= & X^{1}+\varepsilon\pi_{2}\left(Z\right)^{1}\\
V_{\left(Z,\varphi\right)}^{2}\left(\varepsilon\right)= & X^{2}+\varepsilon\left[\pi_{1,2}\left(Z\right)+\pi_{2,1}\left(Z\right)+\varphi\right]+\varepsilon^{2}\pi_{2}\left(Z\right)^{2}.\end{align*}
\end{defn}
\begin{rem}
By proposition \ref{pro:RPAdd}, $V_{\left(Z,\varphi\right)}$ is
actually in $WG\Omega_{p}\left(V\right)$. \end{rem}
\begin{prop}
The above relation $\sim$ is an equivalence relation. \end{prop}
\begin{proof}
This is clear because the relation is defined by an equality.
\end{proof}
We now demonstrate that each equivalence class of curves through $X$,
where equivalence is defined by equal derivative in the sense of our
differential structure on $T^{2}\left(V\right)$, can be associated
to a unique$\left[Z,\varphi\right]$. If we denote the collection
of all such equivalence classes by $T_{X}WG\Omega_{p}$, this will
justify calling $T_{X}WG\Omega_{p}$ the tangent space. 
\begin{thm}
\label{thm:TgtFromCurve}Let $c\left(\varepsilon\right):\left[-\tau,\tau\right]\rightarrow WG\Omega_{p}\left(V\right)$
be such that $c\left(0\right)=X$ and let \[
\left(1,C'\left(0\right)^{1},C\left(0\right)^{2}\right)\]
be its derivative at $0$ in $T^{2}\left(V\right)$ sense. If $c'\left(0\right)^{1}$
has finite $p$-variation and $c'\left(0\right)^{2}$ has finite $\frac{p}{2}$-variation,
there exists a unique equivalence class $\left[Z,\varphi\right]$
such that \begin{align*}
\left.\frac{d}{d\varepsilon}\right|_{\varepsilon=0}c\left(\varepsilon\right)= & \left.\frac{d}{d\varepsilon}\right|_{\varepsilon=0}V_{\left[Z,\varphi\right]}.\end{align*}
\end{thm}
\begin{proof}
If $W$ is a chosen Lyons-Victoir extension of $\left(X^{1},c'\left(0\right)^{1}\right)$
and \begin{align*}
\varphi_{s,t}:= & c'\left(0\right)_{s,t}^{2}-\pi_{1,2}\left(W\right)_{s,t}-\pi_{2,1}\left(W\right)_{s,t},\end{align*}
form $Z$ defined by\begin{align*}
Z^{1}= & \left(X^{1},c'\left(0\right)^{1}\right)\\
Z^{2}= & \left(\begin{array}{cc}
X^{2} & \pi_{1,2}\left(W\right)\\
\pi_{2,1}\left(W\right) & \pi_{2}\left(W\right)^{2}\end{array}\right).\end{align*}
Such $Z$ is multiplicative, since $X$ and $W$ are. Let us verify
that $\varphi$ is additive now. Since the curve $c\left(\varepsilon\right)$
is multiplicative for each $\varepsilon$, we have \begin{align*}
c\left(\varepsilon\right)_{s,u}^{2}= & c\left(\varepsilon\right)_{s,t}^{2}+c\left(\varepsilon\right)_{t,u}^{2}+c\left(\varepsilon\right)_{s,t}^{1}\otimes c\left(\varepsilon\right)_{t.u}^{1}\end{align*}
which implies \begin{align*}
c'\left(0\right)_{s,u}^{2}-c'\left(0\right)_{s,t}^{2}-c'\left(0\right)_{t,u}^{2}= & c'\left(0\right)_{s,t}^{1}\otimes X_{t,u}^{1}+X_{s,t}^{1}\otimes c'\left(0\right)_{t,u}^{1}\end{align*}
but also, by the construction of $W$, we have\begin{align*}
\pi_{1,2}\left(W\right)_{s,u}-\pi_{1,2}\left(W\right)_{s,t}-\pi_{1,2}\left(W\right)_{t,u}= & X_{s,t}^{1}\otimes c'\left(0\right)_{t,u}^{1}\end{align*}
and\begin{align*}
\pi_{2,1}\left(W\right)_{s,u}-\pi_{2,1}\left(W\right)_{s,t}-\pi_{2,1}\left(W\right)_{t,u} & =c'\left(0\right)_{s,t}^{1}\otimes X_{t,u}^{1}.\end{align*}
Hence, \begin{align*}
c'\left(0\right)_{s,u}^{2}-\pi_{1,2}\left(W\right)_{s,u}-\pi_{2,1}\left(W\right)_{s,u}= & c'\left(0\right)_{s,t}^{2}-\pi_{1,2}\left(W\right)_{s,t}-\pi_{2,1}\left(W\right)_{s,t}\\
 & +c'\left(0\right)_{t,u}^{2}-\pi_{1,2}\left(W\right)_{t,u}-\pi_{2,1}\left(W\right)_{t,u}\end{align*}
which shows that $\varphi_{s,t}=c'\left(0\right)_{s,t}^{2}-\pi_{1,2}\left(W\right)_{s,t}-\pi_{2,1}\left(W\right)_{s,t}$
is additive. Then, by definition the associated variational curve
$V_{\left[Z,\varphi\right]}\left(\varepsilon\right)\in WG\Omega_{p}\left(V\right)$
where \begin{align*}
V_{\left[Z,\varphi\right]}\left(\varepsilon\right)^{1}= & X^{1}+\varepsilon c'\left(0\right)^{1}\\
V_{\left[Z,\varphi\right]}\left(\varepsilon\right)^{2}= & X^{2}+\varepsilon\left(\pi_{1,2}\left(W\right)+\pi_{2,1}\left(W\right)+\varphi\right)+\varepsilon^{2}\pi_{2}\left(W\right)^{2}.\end{align*}
And by construction $\left.\frac{d}{d\varepsilon}\right|_{\varepsilon=0}V_{\left[Z,\varphi\right]}=\left.\frac{d}{d\varepsilon}\right|_{\varepsilon=0}c\left(\varepsilon\right)$.
\end{proof}
Let us show that $T_{X}WG\Omega_{p}$ is a linear space.
\begin{lem}
\label{lem:EquivAdd}Given \textup{$Z_{1}=\left(X,Y\right),Z_{2}=\left(X,\tilde{Y}\right)\in G\Omega_{p}\left(V\oplus W\right)$
and} $W\in WG\Omega_{p}\left(W\right)$, a particular choice of Lyons-Victoir
extension of $\left(1,Y^{1}+\tilde{Y}^{1}\right)$, and \[
Z\in C_{0}\left(\Delta_{T};T^{2}\left(V\oplus W\right)\right)\]
defined by\begin{align*}
Z_{s,t}^{1}= & \left(X_{s,t}^{1},Y_{s,t}^{1}+\tilde{Y}_{s,t}^{1}\right)\\
Z_{s,t}^{2}= & \left(\begin{array}{cc}
X_{s,t}^{2} & \sum_{i=1}^{2}\pi_{1,2}\left(Z_{i}\right)_{s,t}\\
\sum_{i=1}^{2}\pi_{2,1}\left(Z_{i}\right)_{s,t} & W_{s,t}^{2}\end{array}\right)\end{align*}

is a weakly geometric $p$-rough path over $V\oplus W$.\end{lem}
\begin{proof}
Consider $\left(1,Y^{1}+\tilde{Y}^{1}\right).$ We can extend this
via the Lyons-Victoir extension to a weakly geometric rough path $W\in WG\Omega_{p}\left(W\right)$
which we write as \[
W=\left(1,Y^{1}+\tilde{Y}^{1},W^{2}\right).\]
By definition,\begin{align*}
\left(Z_{s,t}\otimes Z_{t,u}\right)_{1,1}^{2}= & X_{s,t}^{2}+X_{t,u}^{2}+X_{s,t}^{1}\otimes X_{t,u}^{1}\\
\left(Z_{s,t}\otimes Z_{t,u}\right)_{1,2}^{2}= & \sum_{i=1}^{2}\pi_{1,2}\left(Z_{i}\right)_{s,t}+\sum_{i=1}^{2}\pi_{1,2}\left(Z_{i}\right)_{t,u}+\left(X_{s,t}^{1}\otimes Y_{t,u}^{1}+X_{s,t}^{1}\otimes\tilde{Y}_{t,u}^{1}\right)\\
\left(Z_{s,t}\otimes Z_{t,u}\right)_{2,1}^{2}= & \sum_{i=1}^{2}\pi_{2,1}\left(Z_{i}\right)_{s,t}+\sum_{i=1}^{2}\pi_{2,1}\left(Z_{i}\right)_{t,u}+\left(Y_{s,t}^{1}\otimes X_{t,u}^{1}+\tilde{Y}_{s,t}^{1}\otimes X_{t,u}^{1}\right)\\
\left(Z_{s,t}\otimes Z_{t,u}\right)_{2,2}^{2}= & W_{s,t}^{2}+W_{t,u}^{2}+Y_{s,t}^{1}\otimes Y_{t,u}^{1}+Y_{s,t}^{1}\otimes\tilde{Y}_{t,u}^{1}+\tilde{Y}_{s,t}^{1}\otimes Y_{t,u}^{1}+\tilde{Y}_{s,t}^{1}\otimes\tilde{Y}_{t,u}^{1}\end{align*}
 Since $Z_{i}$ is multiplicative for $i=1,2$ we have\begin{align*}
\left(\begin{array}{cc}
X_{s,t}^{2}+X_{t,u}^{2}+X_{s,t}^{1}\otimes X_{t,u}^{1} & \pi_{1,2}\left(Z_{1}\right)_{s,t}+\pi_{1,2}\left(Z_{1}\right)_{t,u}+X_{s,t}^{1}\otimes Y_{t,u}^{1}\\
\pi_{2,1}\left(Z_{1}\right)_{s,t}+\pi_{2,1}\left(Z_{1}\right)_{t,u}+Y_{s,t}^{1}\otimes X_{t,u}^{1} & Y_{s,t}^{2}+Y_{t,u}^{2}+Y_{s,t}^{1}\otimes Y_{t,u}^{1}\end{array}\right)\\
=\left(\begin{array}{cc}
X_{s,u}^{2} & \pi_{1,2}\left(Z_{1}\right)_{s,u}\\
\pi_{2,1}\left(Z_{1}\right)_{s,u} & Y_{s,u}^{2}\end{array}\right)\end{align*}
and\begin{align*}
\left(\begin{array}{cc}
X_{s,t}^{2}+X_{t,u}^{2}+X_{s,t}^{1}\otimes X_{t,u}^{1} & \pi_{1,2}\left(Z_{2}\right)_{s,t}+\pi_{1,2}\left(Z_{2}\right)_{t,u}+X_{s,t}^{1}\otimes\tilde{Y}_{t,u}^{1}\\
\pi_{2,1}\left(Z_{2}\right)_{s,t}+\pi_{2,1}\left(Z_{2}\right)_{t,u}+\tilde{Y}_{s,t}^{1}\otimes X_{t,u}^{1} & \tilde{Y}_{s,t}^{2}+\tilde{Y}_{t,u}^{2}+\tilde{Y}_{s,t}^{1}\otimes\tilde{Y}_{t,u}^{1}\end{array}\right)\\
=\left(\begin{array}{cc}
X_{s,u}^{2} & \pi_{1,2}\left(Z_{2}\right)_{s,u}\\
\pi_{2,1}\left(Z_{2}\right)_{s,u} & \tilde{Y}_{s,u}^{2}\end{array}\right).\end{align*}
Also, since $W$ is multiplicative, \begin{align*}
W_{s,t}^{2}+W_{t,u}^{2}+\left(Y^{1}+\tilde{Y}^{1}\right)\otimes\left(Y^{1}+\tilde{Y}^{1}\right)= & W_{s,u}.\end{align*}
Putting these expressions together, we get $\left(Z_{s,t}\otimes Z_{t,u}\right)^{2}=Z_{s,u}^{2}$. \end{proof}
\begin{prop}
The tangent space $T_{X}G\Omega_{p}\left(V\right)$ is linear. More
precisely the operations \begin{align*}
\left[Z_{1},\varphi_{1}\right]+\left[Z_{2},\varphi_{2}\right]:= & \left[Z,\varphi_{1}+\varphi_{2}\right],\end{align*}
where\begin{align*}
Z^{1}= & \left(X_{s,t}^{1},\pi_{2}\left(Z_{1}\right)^{1}+\pi_{2}\left(Z_{2}\right)^{1}\right)\\
Z^{2}= & \left(\begin{array}{cc}
X_{s,t}^{2} & \sum_{i=1}^{2}\pi_{1,2}\left(Z_{i}\right)_{s,t}\\
\sum_{i=1}^{2}\pi_{2,1}\left(Z_{i}\right)_{s,t} & W_{s,t}^{2}\end{array}\right)\end{align*}
for $W$ some Lyons-Victoir extension of $\left(1,\pi_{2}\left(Z_{1}\right)^{1}+\pi_{2}\left(Z_{2}\right)^{1}\right)$,
and \begin{align*}
\lambda\left[Z,\varphi\right]:= & \left[Z_{\lambda},\lambda\varphi\right]\end{align*}
where $Z_{\lambda}$ is defined by \begin{align*}
Z_{\lambda}^{1}= & \left(X_{s,t}^{1},\lambda\pi_{2}\left(Z\right)^{1}\right)\\
Z_{\lambda}^{2}= & \left(\begin{array}{cc}
X_{s,t}^{2} & \lambda\pi_{1,2}\left(Z\right)_{s,t}\\
\lambda\pi_{2,1}\left(Z\right)_{s,t} & \lambda^{2}Z_{s,t}^{2}\end{array}\right)\end{align*}
are well defined.\end{prop}
\begin{proof}
That $Z$, as defined above, is in $WG\Omega_{p}\left(V\oplus V\right)$
follows from Lemma \ref{lem:EquivAdd} so $\left[Z,\varphi_{1}+\varphi_{2}\right]$
is indeed an equivalence class. We now show that $\left[Z,\varphi_{1}+\varphi_{2}\right]$
does not depend on the choice of representatives from $\left[Z_{1},\varphi_{1}\right]$
and $\left[Z_{2},\varphi_{2}\right]$. Let $\left(\tilde{Z_{1}},\tilde{\varphi_{1}}\right)$
and $\left(\tilde{Z_{2}},\tilde{\varphi_{2}}\right)$ be some other
representatives from the equivalence classes $\left[Z_{1},\varphi_{1}\right]$
and $\left[Z_{2},\varphi_{2}\right]$ respectively. We form $\left(\tilde{Z},\tilde{\varphi_{1}}+\tilde{\varphi_{2}}\right)$
with \begin{align*}
\tilde{Z}^{1}= & \left(X_{s,t}^{1},\pi_{2}\left(\tilde{Z_{1}}\right)^{1}+\pi_{2}\left(\tilde{Z_{2}}\right)^{1}\right)\\
\tilde{Z}^{2}= & \left(\begin{array}{cc}
X_{s,t}^{2} & \sum_{i=1}^{2}\pi_{1,2}\left(\tilde{Z}_{i}\right)_{s,t}\\
\sum_{i=1}^{2}\pi_{2,1}\left(\tilde{Z_{i}}\right)_{s,t} & W_{s,t}^{2}\end{array}\right).\end{align*}
Then\begin{align*}
V_{\left(\tilde{Z},\tilde{\varphi_{1}}+\tilde{\varphi_{2}}\right)}^{1}\left(\varepsilon\right)= & X^{1}+\varepsilon\left[\pi_{2}\left(\tilde{Z_{1}}\right)^{1}+\pi_{2}\left(\tilde{Z_{2}}\right)^{1}\right]\\
V_{\left(\tilde{Z},\tilde{\varphi_{1}}+\tilde{\varphi_{2}}\right)}^{2}\left(\varepsilon\right)= & X^{2}+\varepsilon\sum_{i=1}^{2}\pi_{1,2}\left(\tilde{Z}_{i}\right)_{s,t}+\varepsilon\sum_{i=1}^{2}\pi_{2,1}\left(\tilde{Z_{i}}\right)_{s,t}+\varepsilon\left(\tilde{\varphi_{1}}+\tilde{\varphi}_{2}\right)+\varepsilon^{2}W_{s,t}^{2}.\end{align*}
Furthermore, since $\left(Z_{i},\varphi_{i}\right)\sim\left(\tilde{Z_{i}},\tilde{\varphi_{i}}\right)$
for $i\in\left\{ 1,2\right\} $, we have \begin{align*}
\left.\frac{d}{d\varepsilon}\right|_{\varepsilon=0}V_{\left(Z_{i},\varphi_{i}\right)}\left(\varepsilon\right)= & \left.\frac{d}{d\varepsilon}\right|_{\varepsilon=0}V_{\left(\tilde{Z_{i}},\tilde{\varphi_{i}}\right)}\left(\varepsilon\right)\end{align*}
which gives \begin{align*}
\pi_{2}\left(Z_{i}\right)^{1}= & \pi_{2}\left(\tilde{Z_{i}}\right)^{1}\end{align*}
and\begin{align*}
\pi_{1,2}\left(Z_{i}\right)+\pi_{2,1}\left(Z_{i}\right)+\varphi_{i}= & \pi_{1,2}\left(\tilde{Z}_{i}\right)+\pi_{2,1}\left(\tilde{Z}_{i}\right)+\tilde{\varphi}_{i}.\end{align*}
Therefore, \begin{align*}
\left.\frac{d}{d\varepsilon}\right|_{\varepsilon=0}V_{\left(\tilde{Z},\tilde{\varphi_{1}}+\tilde{\varphi_{2}}\right)}^{1}\left(\varepsilon\right)= & \left[\pi_{2}\left(\tilde{Z_{1}}\right)^{1}+\pi_{2}\left(\tilde{Z_{2}}\right)^{1}\right]\\
= & \left[\pi_{2}\left(Z_{1}\right)^{1}+\pi_{2}\left(Z_{2}\right)^{1}\right]\\
= & \left.\frac{d}{d\varepsilon}\right|_{\varepsilon=0}V_{\left(Z,\varphi_{1}+\varphi_{2}\right)}^{1}\left(\varepsilon\right)\end{align*}
and also\begin{align*}
\left.\frac{d}{d\varepsilon}\right|_{\varepsilon=0}V_{\left(\tilde{Z},\tilde{\varphi_{1}}+\tilde{\varphi_{2}}\right)}^{2}\left(\varepsilon\right)= & \sum_{i=1}^{2}\pi_{1,2}\left(\tilde{Z}_{i}\right)_{s,t}+\sum_{i=1}^{2}\pi_{2,1}\left(\tilde{Z_{i}}\right)_{s,t}+\tilde{\varphi}_{1}+\tilde{\varphi}_{2}\\
= & \sum_{i=1}^{2}\pi_{1,2}\left(Z_{i}\right)_{s,t}+\sum_{i=1}^{2}\pi_{2,1}\left(Z_{i}\right)_{s,t}+\varphi_{1}+\varphi_{2}\\
= & \left.\frac{d}{d\varepsilon}\right|_{\varepsilon=0}V_{\left(Z,\varphi_{1}+\varphi_{2}\right)}^{2}\left(\varepsilon\right).\end{align*}
Hence $\left(\tilde{Z},\tilde{\varphi_{1}}+\tilde{\varphi_{2}}\right)\sim\left(Z,\varphi_{1}+\varphi_{2}\right)$
and the addition is well defined. For scalar multiplication, $Z_{\lambda}\in WG\Omega_{p}\left(V\right)$
follows from proposition \ref{pro:RPAdd}. The fact that it is independent
of the choice of equivalence class representative is the same in nature
as the proof of this fact for addition.
\end{proof}

\section{Flow Equations}

In this section we investigate the existence and uniqueness of solutions
to flow equations on the space of weakly geometric rough paths $WG\Omega_{p}$
given by \begin{align*}
C'\left(\tau\right)= & F\left(C\left(\tau\right)\right),\mbox{ }C\left(0\right)=X.\end{align*}
We interpret the derivative on the left hand side as the associated
tangent vector in the sense given by theorem \ref{thm:TgtFromCurve}
and $F$ assigns elements of the tangent space at $C\left(\tau\right)$
to each $C\left(\tau\right)$. Key to the constructions in this section
is that although the collection of tangent spaces does not have a
vector bundle structure due to lack of local trivialization, there
is some useful structure which comes from rough paths being imbedded
in the the linear space $C\left(\Delta;T^{2}\left(V\right)\right)$
equipped with the $p$-variation distance. 

Throughout this section, we consider a class of vector fields given
as follows. Suppose we are given two maps $Z:WG\Omega_{p}\left(V\right)\rightarrow WG\Omega_{p}\left(V\oplus V\right)$
and $\varphi:WG\Omega_{p}\left(V\right)\rightarrow WG\Omega_{\frac{p}{2}}\left(V\right)$
such that for each $X$, $\pi_{1}\left(Z\left(X\right)\right)=X$.
Then, for each $X$, we can form the tangent vector $\left[Z\left(X\right),\varphi\left(X\right)\right]$
and so obtain an associated vector field $F\left(X\right)$. We will
use the notation $F_{Z,\varphi}\left(X\right):=\pi_{1,2}\left(Z\left(X\right)\right)+\pi_{2,1}\left(Z\left(X\right)\right)+\varphi\left(X\right)$
together with $F_{Z}\left(X\right)=Z\left(X\right)$ and $F_{\varphi}\left(X\right)=\varphi\left(X\right)$
when we wish to emphasize the tangent field generated by $Z$ and
$\varphi$.

First let us define the following function on the tangent collection
for future clarity of the notation.
\begin{defn}
For $\left[Z,\varphi\right]_{X}\in T_{X}WG\Omega_{p}\left(V\right)$
and $\left[W,\phi\right]_{Y}\in T_{Y}WG\Omega_{p}\left(V\right)$
let\begin{align*}
\tilde{d}_{p}\left(\left[Z,\varphi\right],\left[W,\phi\right]\right):= & \max\left[d_{p}\left(\pi_{1}\left(Z\right),\pi_{1}\left(W\right)\right)\right.\\
 & \sup_{\mathcal{D}}\left\{ \sum_{l}\left|\pi_{2}\left(Z\right)_{t_{l},t_{l+1}}^{1}-\pi_{2}\left(W\right)_{t_{l},t_{l+1}}^{1}\right|^{p}\right\} ^{\frac{1}{p}}\\
 & \sup_{\mathcal{D}}\left\{ \sum_{l}\left|\left[\pi_{1,2}\left(Z\right)+\pi_{2,1}\left(Z\right)+\varphi\right]{}_{t_{l},t_{l+1}}\right.\right.\\
 & -\left.\left.\left.\left[\pi_{1,2}\left(W\right)+\pi_{2,1}\left(W\right)+\phi\right]{}_{t_{l},t_{l+1}}\right|^{\frac{p}{2}}\right\} ^{\frac{2}{p}}\right]\end{align*}
which does not depend on the choice of representative of the equivalence
class.\end{defn}
\begin{prop}
The function $\tilde{d}_{p}$ is a metric on the disjoint union of
the tangent spaces, $TWG\Omega_{p}\left(V\right)$.\end{prop}
\begin{proof}
Non-negativity and symmetry are inherited from the properties of the
norm on $V$. If $\tilde{d}_{p}\left(\left[Z,\varphi\right]_{X},\left[W,\phi\right]_{Y}\right)=0$
then $\pi_{1}\left(Z\right)=\pi_{1}\left(W\right)$ so that $\left[Z,\varphi\right]$
and $\left[W,\phi\right]$ are in the same tangent space. By definition
\ref{def:TgtEquiv}, $\pi_{2}\left(Z\right)^{1}=\pi_{2}\left(W\right)^{1}$
and \[
\pi_{1,2}\left(Z\right)+\pi_{2,1}\left(Z\right)+\varphi=\pi_{1,2}\left(W\right)+\pi_{2,1}\left(W\right)+\phi\]
imply that they are in the same equivalence class. Hence $\tilde{d}_{p}\left(\left[Z,\varphi\right],\left[W,\phi\right]\right)=0$
if and only if $\left[Z,\varphi\right]_{X}=\left[W,\phi\right]_{Y}$. 

For the triangle inequality, consider $\left[S,\sigma\right]\in T_{R}WG\Omega_{p}\left(V\right)$.
The triangle inequality for $d_{p}$ and the Minkowski inequality
give \begin{align*}
d_{p}\left(\pi_{1}\left(Z\right),\pi_{1}\left(W\right)\right) & \leq d_{p}\left(\pi_{1}\left(Z\right),\pi_{1}\left(S\right)\right)+d_{p}\left(\pi_{1}\left(S\right),\pi_{1}\left(W\right)\right),\\
\left\{ \sum_{l}\left|\pi_{2}\left(Z\right)_{t_{l},t_{l+1}}^{1}-\pi_{2}\left(W\right)_{t_{l},t_{l+1}}^{1}\right|^{p}\right\} ^{\frac{1}{p}}\leq & \left\{ \sum_{l}\left|\pi_{2}\left(Z\right)_{t_{l},t_{l+1}}^{1}-\pi_{2}\left(S\right)_{t_{l},t_{l+1}}^{1}\right|^{p}\right\} ^{\frac{1}{p}}\\
 & +\left\{ \sum_{l}\left|\pi_{2}\left(S\right)_{t_{l},t_{l+1}}^{1}-\pi_{2}\left(W\right)_{t_{l},t_{l+1}}^{1}\right|^{p}\right\} ^{\frac{1}{p}},\end{align*}
and \begin{align*}
\left\{ \sum_{l}\left|\left[\pi_{1,2}\left(Z\right)+\pi_{2,1}\left(Z\right)+\varphi\right]{}_{t_{l},t_{l+1}}\right.\right.\\
-\left.\left.\left[\pi_{1,2}\left(W\right)+\pi_{2,1}\left(W\right)+\phi\right]{}_{t_{l},t_{l+1}}\right|^{\frac{p}{2}}\right\} ^{\frac{2}{p}}\leq & \left\{ \sum_{l}\left|\left[\pi_{1,2}\left(Z\right)+\pi_{2,1}\left(Z\right)+\varphi\right]{}_{t_{l},t_{l+1}}\right.\right.\\
 & -\left.\left.\left[\pi_{1,2}\left(S\right)+\pi_{2,1}\left(S\right)+\sigma\right]{}_{t_{l},t_{l+1}}\right|^{\frac{p}{2}}\right\} ^{\frac{2}{p}}\\
 & +\left\{ \sum_{l}\left|\left[\pi_{1,2}\left(S\right)+\pi_{2,1}\left(S\right)+\sigma\right]{}_{t_{l},t_{l+1}}\right.\right.\\
 & -\left.\left.\left[\pi_{1,2}\left(W\right)+\pi_{2,1}\left(W\right)+\phi\right]{}_{t_{l},t_{l+1}}\right|^{\frac{p}{2}}\right\} ^{\frac{2}{p}}.\end{align*}
Therefore, by definition of the supremum as the least upper bound,
in the previous two equations the supremum over all partitions on
the left hand side is less than the sum of the supremums on the right
hand side. Finally, bounding each of \begin{align*}
d_{p}\left(\pi_{1}\left(Z\right),\pi_{1}\left(W\right)\right),\mbox{ }\sup_{\mathcal{D}}\left\{ \sum_{l}\left|\pi_{2}\left(Z\right)_{t_{l},t_{l+1}}^{1}-\pi_{2}\left(W\right)_{t_{l},t_{l+1}}^{1}\right|^{p}\right\} ^{\frac{1}{p}},\mbox{ and}\\
\sup_{\mathcal{D}}\left\{ \sum_{l}\left|\left[\pi_{1,2}\left(Z\right)+\pi_{2,1}\left(Z\right)+\varphi\right]{}_{t_{l},t_{l+1}}-\left[\pi_{1,2}\left(W\right)+\pi_{2,1}\left(W\right)+\phi\right]{}_{t_{l},t_{l+1}}\right|^{\frac{p}{2}}\right\} ^{\frac{2}{p}}\end{align*}
by the sum\begin{align*}
 & \max\left[d_{p}\left(\pi_{1}\left(Z\right),\pi_{1}\left(S\right)\right)\right.\\
 & \sup_{\mathcal{D}}\left\{ \sum_{l}\left|\pi_{2}\left(Z\right)_{t_{l},t_{l+1}}^{1}-\pi_{2}\left(S\right)_{t_{l},t_{l+1}}^{1}\right|^{p}\right\} ^{\frac{1}{p}}\\
 & \sup_{\mathcal{D}}\left\{ \sum_{l}\left|\left[\pi_{1,2}\left(Z\right)+\pi_{2,1}\left(Z\right)+\varphi\right]{}_{t_{l},t_{l+1}}\right.\right.\\
 & -\left.\left.\left.\left[\pi_{1,2}\left(S\right)+\pi_{2,1}\left(S\right)+\sigma\right]{}_{t_{l},t_{l+1}}\right|^{\frac{p}{2}}\right\} ^{\frac{2}{p}}\right]\\
 & +\max\left[d_{p}\left(\pi_{1}\left(S\right),\pi_{1}\left(W\right)\right)\right.\\
 & \sup_{\mathcal{D}}\left\{ \sum_{l}\left|\pi_{2}\left(S\right)_{t_{l},t_{l+1}}^{1}-\pi_{2}\left(W\right)_{t_{l},t_{l+1}}^{1}\right|^{p}\right\} ^{\frac{1}{p}}\\
 & \sup_{\mathcal{D}}\left\{ \sum_{l}\left|\left[\pi_{1,2}\left(S\right)+\pi_{2,1}\left(S\right)+\sigma\right]{}_{t_{l},t_{l+1}}\right.\right.\\
 & -\left.\left.\left.\left[\pi_{1,2}\left(W\right)+\pi_{2,1}\left(W\right)+\phi\right]{}_{t_{l},t_{l+1}}\right|^{\frac{p}{2}}\right\} ^{\frac{2}{p}}\right]\end{align*}
gives the result.
\end{proof}
Now we are able to state our Lipschitz condition on $F$ which allows
existence solutions. 
\begin{defn}
\label{def:Fcond}Let $Z:WG\Omega_{p}\left(V\right)\rightarrow WG\Omega_{p}\left(V\oplus V\right)$
and $\varphi:WG\Omega_{p}\left(V\right)\rightarrow WG\Omega_{\frac{p}{2}}\left(V\right)$
such that $\pi_{1}\left(Z\left(X\right)\right)=X$ and form the vector
field $F$ such that $F\left(X\right)=\left[Z\left(X\right),\varphi\left(X\right)\right]$. 

We will say such an $F$ is locally Lipschitz near $X_{0}$ if there
exist positive constants $r,\mbox{ }C_{1},\mbox{ and}$ $C_{2}$ such
that for all $X,Y\in B_{r}\left(X_{0}\right)$ \begin{align*}
\tilde{d}_{p}\left(F\left(X\right),F\left(Y\right)\right)\leq & C_{1}d_{p}\left(X,Y\right)\\
\tilde{d}_{q}\left(F\left(X\right),F\left(Y\right)\right)\leq & C_{2}d_{p}\left(X,Y\right)\end{align*}
for $q>p$. 

We will say $F$ is globally Lipschitz if the above relations hold
for all $X,Y\in WG\Omega_{p}\left(V\right)$.
\end{defn}
Now let us introduce our concept of a solution to the flow equation
for a non-differentiable curve. We know from the definition of the
equivalence class, that two tangents (at a fixed point $X$) are in
the same equivalence class if the derivatives at $\varepsilon=0$
of their variational curves are equal, i.e. $\left(Z,\varphi\right)\sim\left(\tilde{Z},\tilde{\varphi}\right)$
if and only if $\left.\frac{d}{d\varepsilon}\right|_{\varepsilon=0}V_{\left(Z,\varphi\right)}\left(\varepsilon\right)=\left.\frac{d}{d\varepsilon}\right|_{\varepsilon=0}V_{\left(\tilde{Z},\tilde{\varphi}\right)}\left(\varepsilon\right)$.
Also, using theorem \ref{thm:TgtFromCurve}, we can associate any
differentiable curve $U$ to a unique equivalence class $\left[Z,\varphi\right]$
such that $\left.\frac{d}{d\varepsilon}\right|_{\varepsilon=0}U\left(\varepsilon\right)=\left.\frac{d}{d\varepsilon}\right|_{\varepsilon=0}V_{\left[Z,\varphi\right]}\left(\varepsilon\right)$.
Hence, a curve $U\left(\tau\right)$ differentiable in our sense is
a classical solution to the flow equation if\[
\lim_{\varepsilon\rightarrow0}\frac{U\left(\tau+\varepsilon\right)-U\left(\tau\right)}{\varepsilon}=\lim_{\varepsilon\rightarrow0}\frac{V_{\left[F_{Z}\left(U\left(\tau\right)\right),F_{\varphi}\left(U\left(\tau\right)\right)\right]}\left(\varepsilon\right)-U\left(\tau\right)}{\varepsilon}\]
for each $\tau\in\left[0,T\right]$. Formally rearranging this equation
leaves us with \[
\lim_{\varepsilon\rightarrow0}\frac{U\left(\tau+\varepsilon\right)-V_{\left[F_{Z}\left(U\left(\tau\right)\right),F_{\varphi}\left(U\left(\tau\right)\right)\right]}\left(\varepsilon\right)}{\varepsilon}=0.\]
This second expression can be defined even for some non-differentiable
curves. The natural metric in which this limit should take place is
of course the $p$-variation metric. However, the following definition
uses $q$ (for $q>p$) instead of $p$ due to technical limitations
which will become clear in the proof.
\begin{defn}
\label{def:LightSoln}A continuous curve $U:\left[0,T\right]\rightarrow\Omega_{q}\left(V\right)$
is said to be a solution to the flow equation\begin{align*}
C'\left(\tau\right)= & F\left(C\left(\tau\right)\right)\\
C\left(0\right)= & X\end{align*}
 if $U\left(0\right)=X$ and \begin{align*}
\lim_{h\downarrow0}h^{-1}\left[d_{q}\left(U\left(\tau+h\right),V_{\left[F_{Z}\left(U\left(\tau\right)\right),F_{\varphi}\left(U\left(\tau\right)\right)\right]}\left(h\right)\right)\right]= & 0\end{align*}
for each $\tau\in\left[0,T\right]$ and some $q>p$. 

$U$ is said to be an $\varepsilon$-solution to the flow equation
if $U\left(0\right)=X$ and\begin{align*}
\lim_{h\downarrow0}h^{-1}\left[d_{p}\left(U\left(\tau+h\right),V_{\left[F_{Z}\left(U\left(\tau\right)\right),F_{\varphi}\left(U\left(\tau\right)\right)\right]}\left(h\right)\right)\right]\leq & \varepsilon\end{align*}
for each $\tau\in\left[0,T\right]$. 
\end{defn}
Roughly speaking, an $\varepsilon$-solution is a curve whose tangent
at a point is in some sense close to the tangent assigned by $F$
at that point. Note that if we were in a Banach space case then the
above would be \[
\lim_{h\downarrow0}\left|\frac{U\left(\tau+h\right)-\left[U\left(\tau\right)+hF\left(U\left(\tau\right)\right)\right]}{h}\right|,\]
 i.e. it would be equivalent to the statement: $U'\left(\tau\right)$
exists and equals $F\left(U\left(\tau\right)\right)$.
\begin{rem}
Recall that for a given tangent vector, a variational curve is not
unique. In fact there are infinitely many choices each associated
to a different Lyons-Victoir extension, indeed a variational curve
for a tangent at $X$ has component in $V\otimes V$ given by $V_{\left[Z,\varphi\right]}\left(\varepsilon\right)^{2}=X^{2}+\varepsilon\left(\pi_{1,2}\left(Z\right)+\pi_{2,1}\left(Z\right)+\varphi\right)+\varepsilon^{2}\left(W^{2}\right)$
where $W$ is an extension of $X^{1}+\pi_{2}\left(Z\right)^{1}$.
Therefore, the above definition may depend on the choice of variational
curve which we associate to the tangent $\left[F_{Z}\left(U\left(\tau\right)\right),F_{\varphi}\left(U\left(\tau\right)\right)\right]$.
However, for the class of vector fields we consider, the canonical
choice of $W^{2}$ is given by $\pi_{2}\left(Z\left(U\left(\tau\right)\right)\right)^{2}$.
We will always assume we make this choice and therefore, we refer
to the variational curve rather than a variational curve. 
\end{rem}
In order to obtain solutions, we first construct an approximate solution
given by an Euler approximation scheme. The only issue is correctly
choosing the Euler {}``polygon'' approximation through the variational
curve associated to a tangent vector.
\begin{lem}
\label{lem:EpSolExistence}Let $F$ be locally Lipschitz near $X_{0}$
in the sense of definition \ref{def:Fcond} and define\begin{align*}
M= & \sup_{X\in B_{r}\left(X_{0}\right)}\left[\tilde{d}_{p}\left(F\left(X\right),0\right),\sup_{\mathcal{D}}\left\{ \sum_{l}\left|\pi_{2}\left(Z\left(X\right)\right)_{t_{l},t_{l+1}}^{2}\right|^{\frac{p}{2}}\right\} ^{\frac{2}{p}}\right]\end{align*}
and set $\alpha=\frac{r}{2M}.$ Then for all $\varepsilon>0$ there
exists an $\varepsilon$-solution $U\in C\left(\left[0,\alpha\right]:B_{r}\left(X_{0}\right)\right)$
of the flow equation with initial value $X_{0}\in WG\Omega_{p}\left(V\right)$
satisfying\begin{equation}
d_{p}\left(U_{\tau},U_{\sigma}\right)\leq\left(1+2\alpha\right)M\left|\tau-\sigma\right|,\mbox{ }\forall\sigma,\tau\in\left[0,\alpha\right].\label{eq:EpsSolnLip}\end{equation}
 \end{lem}
\begin{proof}
Let $\varepsilon>0$ be given and partition the interval $\left[0,\alpha\right]$
into \[
0=\tau_{0}<\tau_{1}<\cdots<\tau_{n}=\alpha\]
such that \[
\max_{i}\left|\tau_{i+1}-\tau_{i}\right|\leq\frac{\varepsilon}{\left(C_{1}+2\right)M}.\]
For $\tau\in\left[\tau_{i},\tau_{i+1}\right)$, define inductively
\begin{align*}
U\left(\tau\right)= & V_{\left[F\left(U_{\tau_{i}}\right)\right]}\left(\tau-\tau_{i}\right)\end{align*}
where $V_{\left[F\left(U_{\tau_{i}}\right)\right]}$ is the variational
curve associated to the tangent $F\left(U_{\tau_{i}}\right)$. More
precisely, we first define the points \begin{align*}
U_{\tau_{1}}^{1}= & X_{0}^{1}+\tau_{1}\pi_{2}\left(F\left(X_{0}\right)\right)^{1}\\
U_{\tau_{1}}^{2}= & X_{0}^{2}+\tau_{1}\left[\pi_{1,2}\left(X_{0}\right)+\pi_{2,1}\left(X_{0}\right)+F_{\varphi}\left(X_{0}\right)\right]++\tau_{1}^{2}\pi_{2}\left(Z\left(X_{0}\right)\right)^{2}\end{align*}
and \begin{align*}
U_{\tau_{i+1}}^{1}= & U_{\tau_{i}}^{1}+\left(\tau_{i+1}-\tau_{i}\right)\pi_{2}\left(F\left(U_{\tau_{i}}\right)\right)^{1}\\
U_{\tau_{i+1}}^{2}= & U_{\tau_{i}}^{2}+\left(\tau_{i+1}-\tau_{i}\right)\left[\pi_{1,2}\left(U_{\tau_{i}}\right)+\pi_{2,1}\left(U_{\tau_{i}}\right)+F_{\varphi}\left(U_{\tau_{i}}\right)\right]\\
 & +\left(\tau_{i+1}-\tau_{i}\right)^{2}\pi_{2}\left(Z\left(U_{\tau_{i}}\right)\right)^{2}\end{align*}
for $i\in\left\{ 1,\ldots,n\right\} $ and then the curve is defined
by \begin{align*}
\left(V_{\left[F\left(U_{\tau_{i}}\right)\right]}\left(\tau\right)\right)^{1}= & U_{\tau_{i}}^{1}+\left(\tau-\tau_{i}\right)\pi_{2}\left(F\left(U_{\tau_{i}}\right)\right)^{1}\\
\left(V_{\left[F\left(U_{\tau_{i}}\right)\right]}\left(\tau\right)\right)^{2}= & U_{\tau_{i}}^{2}+\left(\tau-\tau_{i}\right)\left[\pi_{1,2}\left(U_{\tau_{i}}\right)+\pi_{2,1}\left(U_{\tau_{i}}\right)+F_{\varphi}\left(U_{\tau_{i}}\right)\right]\\
 & +\left(\tau-\tau_{i}\right)^{2}\pi_{2}\left(Z\left(U_{\tau_{i}}\right)\right)^{2}\end{align*}
whenever $\tau\in\left[\tau_{i},\tau_{i+1}\right)$. 

The curve thus defined satisfies $U\left(\tau\right)\in B_{r}\left(X_{0}\right)$.
In fact, for $\tau\in\left[0,\tau_{1}\right)$\begin{align*}
\left\Vert U^{1}\left(\tau\right)-X_{0}^{1}\right\Vert = & \tau\left\Vert \pi_{2}\left(F\left(X_{0}\right)\right)^{1}\right\Vert \\
\left\Vert U^{2}\left(\tau\right)-X_{0}^{2}\right\Vert \leq & \tau\left\Vert F_{Z,\varphi}\left(X_{0}\right)\right\Vert +\tau^{2}\left\Vert \pi_{2}\left(F\left(X_{0}\right)\right)^{2}\right\Vert \end{align*}
so that $d_{p}\left(U\left(\tau\right),X_{0}\right)\leq2\tau_{1}M$.
Similarly, for $\tau\in\left[\tau_{i},\tau_{i+1}\right)$\begin{align*}
\left\Vert U^{1}\left(\tau\right)-U^{1}\left(\tau_{i}\right)\right\Vert = & \left(\tau-\tau_{i}\right)\left\Vert \pi_{2}\left(F\left(U\left(\tau_{i}\right)\right)\right)^{1}\right\Vert \\
\left\Vert U^{2}\left(\tau\right)-U^{2}\left(\tau_{i}\right)\right\Vert \leq & \left(\tau-\tau_{i}\right)\left\Vert F_{Z,\varphi}\left(U\left(\tau_{i}\right)\right)\right\Vert +\left(\tau-\tau_{i}\right)^{2}\left\Vert \pi_{2}\left(Z\left(U\left(\tau_{i}\right)\right)\right)^{2}\right\Vert \end{align*}
implies $d_{p}\left(U\left(\tau_{i}\right),U\left(\tau\right)\right)\leq2\left|\tau_{i+1}-\tau_{i}\right|M$.
Hence, for general $\tau\in\left[0,\alpha\right]$,\begin{align*}
d_{p}\left(U\left(\tau\right),X_{0}\right)\leq d_{p}\left(X_{0},U_{1}\right)+ & \sum_{i=1}^{k-1}d_{p}\left(U\left(\tau_{i}\right),U\left(\tau_{i+1}\right)\right)+d_{p}\left(U\left(\tau_{k}\right),U\left(\tau\right)\right)\\
\leq & 2M\tau_{1}+2M\sum_{i=1}^{k-1}\left|\tau_{i+1}-\tau_{i}\right|+2M\left|\tau-\tau_{k}\right|\\
= & 2M\tau\\
\leq & 2M\alpha=r\end{align*}
where $k$ is such that $\tau\in\left[\tau_{k},\tau_{k+1}\right)$
and we have used $\left(\tau-\tau_{1}\right)^{2}\leq\left(\tau-\tau_{1}\right)$
for a small enough partition size (which can be controlled by taking
small enough $\varepsilon$). To see the Lipschitz condition is satisfied,
first take $\tau,\sigma\in\left[\tau_{i},\tau_{i+1}\right)$ for which
we have \begin{align*}
U_{\tau}^{1}-U_{\sigma}^{1}= & \left(\tau-\sigma\right)\pi_{2}\left(F_{Z}\left(U_{\tau_{i}}\right)\right)^{1}\\
U_{\tau}^{2}-U_{\sigma}^{2}= & \left(\tau-\sigma\right)F_{Z,\varphi}\left(U_{\tau_{i}}\right)\\
 & +\left[\left(\tau-\sigma\right)+2\left(\sigma-\tau_{i}\right)\right]\left(\tau-\sigma\right)\pi_{2}\left(Z\left(U_{\tau_{i}}\right)\right)^{2}\end{align*}
so that, by homogeneity of the metric, $d_{p}\left(U_{\tau},U_{\sigma}\right)\leq CM\left|\tau-\sigma\right|$.
For general $\tau,\sigma\in\left[0,\alpha\right]$ such that $\sigma\in\left[\tau_{m},\tau_{m+1}\right)$,
$\tau\in\left[\tau_{n},\tau_{n+1}\right)$, we have \begin{align*}
U_{\tau}^{1}-U_{\sigma}^{1}= & \sum_{i=m}^{n}\int_{\sigma}^{\tau}\pi_{2}\left(F\left(U_{\tau_{i}}\right)\right)^{1}\chi_{\left[\tau_{i},\tau_{i+1}\right)}\left(\gamma\right)d\gamma\\
U_{\tau}^{2}-U_{\sigma}^{2}= & \sum_{i=m}^{n}\int_{\sigma}^{\tau}\left[F_{Z,\varphi}\left(U_{\tau_{i}}\right)+2\left(\gamma-\tau_{i}\right)\pi_{2}\left(Z\left(U_{\tau_{i}}\right)\right)^{2}\right]\chi_{\left[\tau_{i},\tau_{i+1}\right)}\left(\gamma\right)d\gamma\end{align*}
where we use the Bochner integral in the Banach spaces $\left(V,\left\Vert \cdot\right\Vert _{V}\right)$
and $\left(V\otimes V,\left\Vert \cdot\right\Vert _{V\otimes V}\right)$.
Consequently, \begin{align*}
\left\Vert \left[U_{\tau}^{1}-U_{\sigma}^{1}\right]\right\Vert \leq & \max_{i}\left\Vert \pi_{2}\left(F\left(U_{\tau_{i}}\right)\right)^{1}\right\Vert \left|\tau-\sigma\right|\\
\left\Vert \left[U_{\tau}^{2}-U_{\sigma}^{2}\right]\right\Vert \leq & \max_{i}\left[\left\Vert F_{Z,\varphi}\left(U_{\tau_{i}}\right)\right\Vert +2\alpha\left\Vert \pi_{2}\left(Z\left(U_{\tau_{i}}\right)\right)^{2}\right\Vert \right]\left|\tau-\sigma\right|\end{align*}
which implies\[
d_{p}\left(U_{\tau},U_{\sigma}\right)\leq\left(1+2\alpha\right)M\left|\tau-\sigma\right|.\]

Now let us show that $U$ is indeed an $\varepsilon$-solution. Whenever
$\tau\in\left[\tau_{i},\tau_{i+1}\right)$, for sufficiently small
$h$, we also have $\left(\tau+h\right)\in\left[\tau_{i},\tau_{i+1}\right)$
and hence \begin{align*}
U\left(\tau+h\right)^{1}= & U_{\tau_{i}}^{1}+\left(\tau+h-\tau_{i}\right)\pi_{2}\left(F\left(U_{\tau_{i}}\right)\right)^{1}\\
U\left(\tau+h\right)^{2}= & U_{\tau_{i}}^{2}+\left(\tau+h-\tau_{i}\right)F_{Z,\varphi}\left(U_{\tau_{i}}\right)\\
 & +\left(\tau+h-\tau_{i}\right)^{2}\pi_{2}\left(Z\left(U_{\tau_{i}}\right)\right)^{2}\end{align*}
and\begin{align*}
V_{F\left(U_{\tau}\right)}\left(h\right)^{1}= & U_{\tau}^{1}+h\pi_{2}\left(F\left(U_{\tau}\right)\right)^{1}\\
V_{F\left(U)_{\tau}\right)}\left(h\right)^{2}= & U_{\tau}^{2}+hF_{Z,\varphi}\left(U_{\tau}\right)+h^{2}\pi_{2}\left(Z\left(U_{\tau}\right)\right)^{2}.\end{align*}
This shows, \begin{align*}
U\left(\tau+h\right)^{1}-V_{F\left(U_{\tau}\right)}\left(h\right)^{1}= & \left(U_{\tau_{i}}^{1}-U_{\tau}^{1}\right)+\left(\tau-\tau_{i}\right)\pi_{2}\left(F\left(U_{\tau_{i}}\right)\right)^{1}\\
 & +h\left(\pi_{2}\left(F\left(U_{\tau_{i}}\right)\right)^{1}-\pi_{2}\left(F\left(U_{\tau}\right)\right)^{1}\right)\end{align*}
and \begin{align*}
U\left(\tau+h\right)^{2}-V_{F\left(U_{\tau}\right)}\left(h\right)^{2}= & \left(U_{\tau_{i}}^{2}-U_{\tau}^{2}\right)+\left(\tau-\tau_{i}\right)F_{Z,\varphi}\left(U_{\tau_{i}}\right)\\
 & +h\left(F_{Z,\varphi}\left(U_{\tau_{i}}\right)-F_{Z,\varphi}\left(U_{\tau}\right)\right)\\
 & \left[\left(\tau-\tau_{i}\right)^{2}+2h\left(\tau-\tau_{i}\right)\right]\pi_{2}\left(Z\left(U_{\tau_{i}}\right)\right)^{2}\\
 & +h^{2}\left[\pi_{2}\left(Z\left(U_{\tau_{i}}\right)\right)^{2}-\pi_{2}\left(Z\left(U_{\tau}\right)\right)^{2}\right].\end{align*}
Since, $\tau_{i}<\tau<\tau_{i+1}$ \begin{align*}
U_{\tau}^{1}= & U_{\tau_{i}}^{1}+\left(\tau-\tau_{i}\right)\pi_{2}\left(F\left(U_{\tau_{i}}\right)\right)^{1}\\
U_{\tau}^{2}= & U_{\tau_{i}}^{2}+\left(\tau-\tau_{i}\right)F_{Z,\varphi}\left(F\left(U_{\tau_{i}}\right)\right)\\
 & +\left(\tau-\tau_{i}\right)^{2}\pi_{2}\left(Z\left(U_{\tau_{i}}\right)\right)^{2}\end{align*}
and so\begin{align*}
\left(U_{\tau_{i}}^{1}-U_{\tau}^{1}\right)+\left(\tau-\tau_{i}\right)\pi_{2}\left(F\left(U_{\tau_{i}}\right)\right)^{1}= & 0\\
\left(U_{\tau_{i}}^{2}-U_{\tau}^{2}\right)+\left(\tau-\tau_{i}\right)F_{Z,\varphi}\left(U_{\tau_{i}}\right)+\left(\tau-\tau_{i}\right)^{2}\pi_{2}\left(Z\left(U_{\tau_{i}}\right)\right)^{2} & =0.\end{align*}
Therefore, \begin{align}
h^{-1}\left(U\left(\tau+h\right)-V_{F\left(U_{\tau}\right)}\left(h\right)\right)^{1}= & \left(\pi_{2}\left(F\left(U_{\tau_{i}}\right)\right)^{1}-\pi_{2}\left(F\left(U_{\tau}\right)\right)^{1}\right)\label{eq:EpsSolEqn1}\\
h^{-1}\left(U\left(\tau+h\right)-V_{F\left(U_{\tau}\right)}\left(h\right)\right)^{2}= & \left(F_{Z,\varphi}\left(U_{\tau_{i}}\right)-F_{Z,\varphi}\left(U_{\tau}\right)\right)\label{eq:EpsSolEqn2}\\
 & +2\left(\tau-\tau_{i}\right)\pi_{2}\left(Z\left(U_{\tau_{i}}\right)\right)^{2}\nonumber \\
 & +h\left[\pi_{2}\left(Z\left(U_{\tau_{i}}\right)\right)^{2}-\pi_{2}\left(Z\left(U_{\tau}\right)\right)^{2}\right]\nonumber \end{align}
and by the Lipschitz property of $F$ and homogeneity of the metric,
\begin{align*}
h^{-1}d_{p}\left(U\left(\tau+h\right),V_{F\left(U_{\tau}\right)}\left(h\right)\right)\leq & C_{1}d_{p}\left(U_{\tau_{i}},U_{\tau}\right)\\
 & +2M\left|\tau-\tau_{i}\right|+2hM.\end{align*}
Also, \begin{align*}
U_{\tau}^{1}= & U_{\tau_{i}}^{1}+\left(\tau-\tau_{i}\right)\pi_{2}\left(F_{Z}\left(U_{\tau_{i}}\right)\right)^{1}\\
U_{\tau}^{2}= & U_{\tau_{i}}^{2}+\left(\tau-\tau_{i}\right)F_{Z,\varphi}\left(U_{\tau_{i}}\right)+\left(\tau-\tau_{i}\right)^{2}\pi_{2}\left(Z\left(U_{\tau_{i}}\right)\right)^{2}\end{align*}
\begin{align*}
d_{p}\left(U_{\tau_{i}},U_{\tau}\right)\leq & M\left|\tau-\tau_{i}\right|.\end{align*}
This implies\begin{align*}
h^{-1}d_{p}\left(U\left(\tau+h\right),V_{F\left(U_{\tau}\right)}\left(h\right)\right)\leq & \left(C_{1}+2\right)M\left|\tau-\tau_{i}\right|+2hM\\
\leq & \varepsilon+2hM\end{align*}
so letting $h\rightarrow0$ gives the result.\end{proof}
\begin{rem}
If the map $Z$ which defines $F$ also satisfies the Lipschitz property,
then the interval on which the $\varepsilon$-solution is defined
has length greater than $\frac{1}{C}$ where $C$ depends only on
the Lipschitz constants and the initial data. Indeed, for all $X\in B_{r}\left(X_{0}\right)$,
$\tilde{d}_{p}\left(F\left(X\right),0\right)\leq C_{1}r+C_{2}$ and
$d_{p}\left(Z\left(X\right),0\right)\leq C_{3}r+C_{4}$ where the
$C_{i}$ are Lipschitz constants or the distance of $X_{0}$ from
$0$. Hence, $\alpha=\frac{r}{M}\geq\frac{r}{C_{i}r+C_{i+1}}$ and
choosing $r\geq1$ gives $\alpha\geq\frac{1}{C_{i}+C_{i+1}}$.
\end{rem}
Consequently, we have an approximate solution for any level of closeness
on a fixed interval which is independent of how close the approximation
is. From here, the Lipschitz property of $U$ given by equation (\ref{eq:EpsSolnLip})
means that a family consisting of $\varepsilon_{n}$ approximations
where $\varepsilon_{n}\rightarrow0$ as $n\rightarrow\infty$ is equicontinuous.
Indeed we have the following.
\begin{lem}
\label{lem:CgtEpsSoln}Let $\left\{ U^{n}\right\} $ be a sequence
of paths in $C\left(\left[0,T\right],WG\Omega_{p}\right)$ such that
for each $n$, $U^{n}$ is a $\varepsilon_{n}$-solution to the flow
equation where $\varepsilon_{n}\rightarrow0$ as $n\rightarrow\infty$.
Then the collection $\left\{ U^{n}\right\} $ has a sub-sequence $\left\{ U^{n_{k}}\right\} $
which converges in $C\left(\left[0,T\right],\Omega_{q}\right)$ equipped
with uniform topology, for all $q>p$.\end{lem}
\begin{proof}
This follows from the Ascoli-Arzela theorem for metric spaces and
theorem\ref{thm:ExtrinCompac}.
\end{proof}
A natural question to consider is whether or not the limit of $\varepsilon_{n}$
solutions given by the preceding lemma is a solution. The proof that
this is the case relies on the vector field also being Lipschitz in
the $q$ variation sense and follows the same arguments as given in
lemma \ref{lem:EpSolExistence}.
\begin{lem}
\label{lem:LimIsSoln}Given a convergent sub-sequence $\left\{ U^{n}\right\} $
of $\varepsilon_{n}$-solutions of the flow equation, the limit $U\in C\left(\left[0,\alpha\right];\Omega_{q}\right)$
is a solution of the flow equation.\end{lem}
\begin{proof}
Let us first consider the distance involving the approximating sequence,
i.e. $d_{q}\left(U^{n}\left(\tau+h\right),V_{F\left(U^{n}\left(\tau\right)\right)}\left(h\right)\right).$
By equations (\ref{eq:EpsSolEqn1}) and (\ref{eq:EpsSolEqn2}) in
the proof of lemma \ref{lem:EpSolExistence}, we have for $\tau\in\left[\tau_{i},\tau_{i+1}\right)$
and sufficiently small $h$

\begin{align*}
h^{-1}\left(U^{n}\left(\tau+h\right)-V_{F\left(U_{\tau}^{n}\right)}\left(h\right)\right)^{1}= & \left(\pi_{2}\left(F\left(U_{\tau_{i}}^{n}\right)\right)^{1}-\pi_{2}\left(F\left(U_{\tau}^{n}\right)\right)^{1}\right)\\
h^{-1}\left(U^{n}\left(\tau+h\right)-V_{F\left(U_{\tau}^{n}\right)}\left(h\right)\right)^{2}= & \left(F_{Z,\varphi}\left(U_{\tau_{i}}^{n}\right)-F_{Z,\varphi}\left(U_{\tau}^{n}\right)\right)\\
 & +2\left(\tau-\tau_{i}\right)\pi_{2}\left(Z\left(U_{\tau_{i}}^{n}\right)\right)^{2}\\
 & +h\left[\pi_{2}\left(Z\left(U_{\tau_{i}}^{n}\right)\right)^{2}-\pi_{2}\left(Z\left(U_{\tau}^{n}\right)\right)^{2}\right].\end{align*}
Letting \[
M^{*}=\sup_{X\in B_{r}\left(X_{0}\right)}\left[\tilde{d}_{q}\left(F\left(X\right),0\right),\sup_{\mathcal{D}}\left\{ \sum_{l}\left|\pi_{2}\left(Z\left(X\right)\right)_{t_{l},t_{l+1}}^{2}\right|^{\frac{q}{2}}\right\} ^{\frac{2}{q}}\right]\]
and, by using the property that $F$ is also Lipschitz in the $q$
variation sense, \begin{align*}
h^{-1}d_{q}\left(U^{n}\left(\tau+h\right),V_{F\left(U_{\tau}^{n}\right)}\left(h\right)\right)\leq & C_{2}d_{p}\left(U_{\tau_{i}},U_{\tau}\right)\\
 & +2M^{*}\left|\tau-\tau_{i}\right|+2hM^{*}\\
\leq & \left(C_{2}+2\right)M^{*}\left|\tau-\tau_{i}\right|+2hM^{*}\\
\leq & \left(C_{2}+2\right)M^{*}\frac{\varepsilon_{n}}{\left(C_{1}+2M\right)}+2hM^{*}.\end{align*}
Hence, \[
\lim_{h\rightarrow0}h^{-1}d_{p}\left(U^{n}\left(\tau+h\right),V_{F\left(U_{\tau}^{n}\right)}\left(h\right)\right)=\frac{\left(C_{2}+2\right)M^{*}}{C_{1}+2M}\varepsilon_{n}\]
and so letting $n\rightarrow\infty$ gives the result.\end{proof}
\begin{lem}
\label{lem:Uniqueness}If a solution exists, then it is unique.\end{lem}
\begin{proof}
Suppose there exist two solutions $U$ and $\tilde{U}$ of the flow
equation with the same initial data. Consider the function \[
g\left(\tau\right):=d_{q}\left(U\left(\tau\right),\tilde{U}\left(\tau\right)\right).\]
Now, \begin{align*}
\frac{g\left(\tau+\varepsilon\right)-g\left(\tau\right)}{\varepsilon}\leq & \frac{d_{q}\left(U\left(\tau+\varepsilon\right),V_{F\left(U\left(\tau\right)\right)}\left(\varepsilon\right)\right)}{\varepsilon}+\frac{d_{q}\left(V_{F\left(\tilde{U}\left(\tau\right)\right)}\left(\varepsilon\right),\tilde{U}\left(\tau+\varepsilon\right)\right)}{\varepsilon}\\
 & +\frac{d_{q}\left(V_{F\left(U\left(\tau\right)\right)}\left(\varepsilon\right),V_{F\left(\tilde{U}\left(\tau\right)\right)}\left(\varepsilon\right)\right)-d_{q}\left(U\left(\tau\right),\tilde{U}\left(\tau\right)\right)}{\varepsilon}.\end{align*}
The first two terms on the right hand side tend to zero as $\varepsilon\rightarrow0$
as $U$ and $\tilde{U}$ are solutions, so let us examine the final
term. Recall that \begin{align*}
V_{F\left(U\left(\tau\right)\right)}\left(\varepsilon\right)^{1}= & U^{1}\left(\tau\right)+\varepsilon\pi_{2}\left(F\left(U\left(\tau\right)\right)\right)^{1}\\
V_{F\left(U\left(\tau\right)\right)}\left(\varepsilon\right)^{2}= & U^{2}\left(\tau\right)+\varepsilon F_{Z,\varphi}\left(U\left(\tau\right)\right)+\varepsilon^{2}\pi_{2}\left(Z\left(U\left(\tau\right)\right)\right)^{2}\end{align*}
so that by the Lipschitz property of $F$ and the homogeneity of the
distance, \begin{align*}
d_{q}\left(V_{F\left(U\left(\tau\right)\right)}\left(\varepsilon\right),V_{F\left(\tilde{U}\left(\tau\right)\right)}\left(\varepsilon\right)\right)\leq & d_{q}\left(U\left(\tau\right),\tilde{U}\left(\tau\right)\right)+C\varepsilon d_{q}\left(U\left(\tau\right),\tilde{U}\left(\tau\right)\right)\\
 & +\varepsilon^{2}d_{q}\left(Z\left(U\left(\tau\right)\right),Z\left(\tilde{U}\left(\tau\right)\right)\right).\end{align*}
Therefore, $D^{+}g\left(\tau\right)\leq Cg\left(\tau\right)$ where
$D^{+}$ indicates the upper Dini derivative\[
\limsup_{\varepsilon\downarrow0}\frac{g\left(\tau+\varepsilon\right)-g\left(\tau\right)}{\varepsilon}.\]
Now $g\left(\tau\right)$ is bounded by the upper Riemann integral
of the Dini derivative so that\begin{align*}
g\left(\tau\right) & \leq\bar{\int_{0}^{\tau}}D^{+}g\left(\sigma\right)d\sigma\\
 & \leq\bar{C\int_{0}^{\tau}}g\left(\sigma\right)d\sigma\\
 & =C\int_{0}^{\tau}g\left(\sigma\right)d\sigma\end{align*}
and an application of Gronwall's inequality gives $g\left(\tau\right)=0$. 

Combining the above results, we have the following.\end{proof}
\begin{thm}
\label{thm:LocalExist}If $F$ is a locally Lipschitz near $X_{0}$
vector field on $WG\Omega_{p}$, then there exists a unique solution
$U:\left[0,\alpha\right]\rightarrow\Omega_{q}\left(V\right)$ to the
flow equation for $q>p$.\end{thm}
\begin{rem}
Using the preceding arguments we would not necessarily have a global
solution even for a globally Lipschitz vector field. This is because
in our definition we do not assume a Lipschitz type condition on the
projection $\pi_{2}\left(Z\left(\cdot\right)\right)^{2}$ which appears
in the definition of the length of the interval. If however we impose
the slightly stronger condition that both the maps $Z$ and $\varphi$
are Lipschitz, then the induced vector field will be Lipschitz and,
moreover, the length of the interval where the solution is defined
is larger than a constant which depends only on the Lipschitz constants
of $Z$ and $\varphi$. In this setting, we have global solutions.\end{rem}
\begin{thm}
\label{thm:GlobalSoln}Let $Z:WG\Omega_{p}\left(V\right)\rightarrow WG\Omega_{p}\left(V\oplus V\right)$
and $\varphi:WG\Omega_{p}\left(V\right)\rightarrow WG\Omega_{\frac{p}{2}}\left(V\right)$
be maps such that $\pi_{1}\left(Z\left(X\right)\right)=X$. If there
exist constants $C_{i}$ for $i=\left\{ 1,\ldots,4\right\} $ such
that \begin{align*}
d_{p}\left(Z\left(X\right),Z\left(Y\right)\right)\leq & C_{1}d_{p}\left(X,Y\right)\\
d_{\frac{p}{2}}\left(\varphi\left(X\right),\varphi\left(Y\right)\right)\leq & C_{2}d_{p}\left(X,Y\right)\end{align*}
and \begin{align*}
d_{q}\left(Z\left(X\right),Z\left(Y\right)\right)\leq & C_{3}d_{q}\left(X,Y\right)\\
d_{\frac{q}{2}}\left(\varphi\left(X\right),\varphi\left(Y\right)\right)\leq & C_{4}d_{q}\left(X,Y\right)\end{align*}
then there exists a global solution $U:\left[0,\infty\right]\rightarrow\Omega_{q}\left(V\right)$
to the flow equation with initial data $X_{0}$ for the vector field
induced by $Z$ and $\varphi$ for $q>p$.\end{thm}
\begin{proof}
If $F$ denotes the vector field induced by $Z$ and $\varphi$, the
above Lipschitz conditions imply\begin{align*}
\tilde{d}_{p}\left(F\left(X\right),F\left(Y\right)\right) & \leq\max\left[1,C_{1}+C_{2}\right]d_{p}\left(X,Y\right)\end{align*}
and\begin{align*}
\tilde{d}_{q}\left(F\left(X\right),F\left(Y\right)\right) & \leq\max\left[1,C_{3}+C_{4}\right]d_{q}\left(X,Y\right).\end{align*}
Letting\begin{align*}
C_{5}= & \max\left[1,C_{1}+C_{2}\right]\\
C_{6}= & \max\left[\tilde{d}_{p}\left(F\left(X_{0}\right),0\right),d_{p}\left(Z\left(X_{0}\right),0\right)\right],\end{align*}
we establish the bounds \begin{align*}
\tilde{d}_{p}\left(F\left(X\right),0\right)\leq & C_{5}d_{p}\left(X,X_{0}\right)+\tilde{d}_{p}\left(F\left(X_{0}\right),0\right)\end{align*}
and \begin{align*}
\sup_{\mathcal{D}}\left\{ \sum_{l}\left|\pi_{2}\left(Z\left(X\right)\right)_{t_{l},t_{l+1}}^{2}\right|^{\frac{p}{2}}\right\} ^{\frac{2}{p}}\leq & d_{p}\left(Z\left(X\right),0\right)\\
\leq & C_{1}d_{p}\left(X,X_{0}\right)+d_{p}\left(Z\left(X_{0}\right),0\right).\end{align*}
For $r_{1}>0$, set\begin{align*}
 & M_{1}=\sup_{X\in B_{r_{1}}\left(X_{0}\right)}\left[\tilde{d}_{p}\left(F\left(X\right),0\right),\sup_{\mathcal{D}}\left\{ \sum_{l}\left|\pi_{2}\left(Z\left(X\right)\right)_{t_{l},t_{l+1}}^{2}\right|^{\frac{p}{2}}\right\} ^{\frac{2}{p}}\right]\\
 & \leq C_{5}r_{1}+C_{6}\end{align*}
and apply lemma \ref{lem:EpSolExistence} so that we have a local
$\varepsilon$-solution $U_{1}^{\varepsilon}\left(\cdot\right)$ on
$\left[0,\alpha_{1}\right]$ for $\alpha_{1}=\frac{r_{1}}{2M_{1}}$.
We now apply the lemma again for initial point $U_{1}^{\varepsilon}\left(\alpha_{1}\right)$,
$r_{2}>0$. To bound $M_{2}$, we derive the estimates \begin{align*}
\tilde{d}_{p}\left(F\left(X\right),0\right)\leq & C_{5}d_{p}\left(X,U_{1}^{\varepsilon}\left(\alpha_{1}\right)\right)+\tilde{d}_{p}\left(F\left(U_{1}^{\varepsilon}\left(\alpha_{1}\right)\right),0\right)\\
\leq & C_{5}\left[d_{p}\left(X,U_{1}^{\varepsilon}\left(\alpha_{1}\right)\right)+d_{p}\left(U_{1}^{\varepsilon}\left(\alpha_{1}\right),X_{0}\right)\right]+\tilde{d}_{p}\left(F\left(X_{0}\right),0\right)\end{align*}
and\begin{align*}
\sup_{\mathcal{D}}\left\{ \sum_{l}\left|\pi_{2}\left(Z\left(X\right)\right)_{t_{l},t_{l+1}}^{2}\right|^{\frac{p}{2}}\right\} ^{\frac{2}{p}}\leq & d_{p}\left(Z\left(X\right),0\right)\\
\leq & C_{1}\left[d_{p}\left(X,U_{1}^{\varepsilon}\left(\alpha_{1}\right)\right)+d_{p}\left(U_{1}^{\varepsilon}\left(\alpha_{1}\right),X_{0}\right)\right]\\
 & +d_{p}\left(Z\left(X_{0}\right),0\right).\end{align*}
Therefore, as $d_{p}\left(U_{1}^{\varepsilon}\left(\alpha_{1}\right),X_{0}\right)\leq r_{1}$,
\begin{align*}
 & M_{2}=\sup_{X\in B_{r_{2}}\left(U_{1}\left(\alpha_{1}\right)\right)}\left[\tilde{d}_{p}\left(F\left(X\right),0\right),\sup_{\mathcal{D}}\left\{ \sum_{l}\left|\pi_{2}\left(Z\left(X\right)\right)_{t_{l},t_{l+1}}^{2}\right|^{\frac{p}{2}}\right\} ^{\frac{2}{p}}\right]\\
 & \leq C_{5}\left(r_{1}+r_{2}\right)+C_{6}\end{align*}
and we get an $\varepsilon$-solution, $U_{2}^{\varepsilon}$ defined
on $\left[0,\alpha_{2}\right]$ where $U_{2}^{\varepsilon}\left(0\right)=U_{1}^{\varepsilon}\left(\alpha_{1}\right)$
and \begin{align*}
\alpha_{2}= & \frac{r_{2}}{2M_{2}}\\
\geq & \frac{1}{2C_{5}\left(\frac{r_{1}+r_{2}}{r_{2}}\right)+\frac{C_{6}}{r_{3}}}.\end{align*}
If we repeat this process $n$ times, then for each $n$ we obtain
an $\varepsilon$-solution, $U_{n}^{\varepsilon}$ defined on interval
of length \begin{align*}
\alpha_{n}\geq & \frac{1}{2C_{5}\left(\frac{\sum_{i=1}^{n-1}r_{i}}{r_{n}}\right)+\frac{C_{6}}{r_{n}}}.\end{align*}
We need a lower bound for $\alpha_{n}$ independent of $n$. If we
choose $r_{i}=e^{i}$ then \begin{align*}
\frac{\sum_{i=1}^{n-1}r_{i}}{r_{n}}= & e^{1-n}+e^{2-n}+\cdots+e^{-1}\\
= & \sum_{j=1}^{n-1}e^{-j}\end{align*}
and by the ratio test, the series $\sum_{j=1}^{\infty}e^{-j}$ converges
to some value $C$. Then with the above choice of $r_{i}$, \begin{align*}
\alpha_{n}\geq & \frac{1}{2C_{5}\left(\frac{\sum_{i=1}^{n-1}r_{i}}{r_{n}}\right)+\frac{C_{6}}{r_{n}}}\\
\geq & \frac{1}{2C_{5}C+\frac{C_{6}}{e}}.\end{align*}
This means that we can repeat the procedure indefinitely to construct
$\varepsilon$-solutions $U_{n}^{\varepsilon}$ each defined on the
interval $\left[0,\alpha_{n}\right]$, with $U_{n}^{\varepsilon}\left(0\right)=U_{n-1}^{\varepsilon}\left(\alpha_{n-1}\right)$
and the total length of the intervals is infinite. 

From here, we follow the same arguments as in the local solution case.
Take $\left(\varepsilon_{m}\right)_{m\geq0}$ such that $\varepsilon_{m}\rightarrow0$
as $m\rightarrow\infty$. By lemma \ref{lem:CgtEpsSoln}, each $U_{n}^{\varepsilon_{m}}$
has a subsequence which converges in $C\left(\left[0,\alpha_{n}\right],WG\Omega_{q}\left(V\right)\right)$
equipped with the uniform norm to $U_{n}$. For $U_{1}^{\varepsilon}$,
take the convergent subsequence $\left\{ U_{1}^{\varepsilon_{l_{1}}}\right\} $
and for the next step, take a convergent subsequence $\left\{ U_{2}^{\varepsilon_{l_{2}}}\right\} $
of $\left\{ U_{2}^{\varepsilon_{l_{1}}}\right\} $, etc. The indicies
of the subsequences are given by $\left\{ \varepsilon_{l_{1}}\right\} \supseteq\left\{ \varepsilon_{l_{2}}\right\} \supseteq\cdots$
so if we take the diagonal subsequence whose indicies $\varepsilon_{l}$
are given by $\varepsilon_{l}=\varepsilon_{l_{l}}$ we have simultaneous
convergence. Since we have uniform, and therefore pointwise convergence,
$U_{n}^{\varepsilon_{_{l}}}\left(\tau\right)\rightarrow U_{n}\left(\tau\right)$
for each $\tau$ and in particular\begin{align*}
U_{n}^{\varepsilon_{l}}\left(0\right)= & \lim_{l\rightarrow\infty}U_{n}^{\varepsilon_{l}}\left(0\right)\\
= & \lim_{l\rightarrow\infty}U_{n-1}^{\varepsilon_{l}}\left(\alpha_{n-1}\right)\\
= & U_{n-1}\left(\alpha_{n-1}\right).\end{align*}
We can put these solutions together to form $\hat{U}:\left[0,\infty\right)\rightarrow\Omega_{q}\left(V\right)$
defined by \[
\hat{U}\left(\tau\right)=\begin{cases}
U_{1}\left(\tau\right) & \mbox{ for }\tau\in\left[0,\alpha_{1}\right]\\
U_{2}\left(\tau-\alpha_{1}\right) & \mbox{ for }\tau\in\left(\alpha_{1},\alpha_{1}+\alpha_{2}\right]\\
\vdots & \vdots\end{cases}\]
which is then a unique global solution by lemmas \ref{lem:LimIsSoln}
and \ref{lem:Uniqueness}. 
\end{proof}

\section*{Appendix}

As we make heavy use of the Lyons-Victoir extension, we present here
a small extension of the result for paths in $\mathbb{R}^{d}$. We
remark that it may be used to define a mapping from a curve of rough
paths $X\left(\varepsilon\right)$ to another rough path which can
be viewed as $\int X\left(\varepsilon\right)d\varepsilon$. The construction
follows closely the proof of the $\mathbb{R}^{d}$ case extension
in the Lyons-Victoir paper \cite{LyonsVicExten}.

Let $X\left(\varepsilon\right)$ be a path taking values in $WG\Omega_{p}\left(\mathbb{R}^{d}\right)$.
If $X\left(\varepsilon\right)$ takes its values only in smooth rough
paths then we could define the integral $X^{\mu}$ via the Bochner
integral:\begin{align*}
X_{s,t}^{\mu,1}= & \int_{0}^{1}\left[\int_{s<u<t}dx_{u}\left(\varepsilon\right)\right]\mu\left(d\varepsilon\right)\\
X_{s,t}^{\mu,2}= & \int_{0}^{1}\int_{0}^{1}\left[\int_{s<u_{1}<u_{2}<t}dx_{u_{1}}\left(\varepsilon\right)\otimes dx_{u_{2}}\left(\delta\right)\right]\mu\left(d\varepsilon\right)\mu\left(d\delta\right).\end{align*}
The term in brackets in the definition of $X_{s,t}^{\mu,2}$ can be
realized for non smooth rough paths as $\pi_{1,2}\left(Z\right)$
where $Z$ is a rough path in $WG\Omega_{p}\left(\mathbb{R}^{d}\oplus\mathbb{R}^{d}\right)$
which extends $\left(X\left(\varepsilon\right),X\left(\delta\right)\right)$.
Therefore, in order to define the integral, we provide some conditions
which allow extension of $\left(X\left(\varepsilon\right),X\left(\delta\right)\right)$
to a $Z$ which we can integrate in $T^{2}\left(\mathbb{R}^{d}\right)$. 

Suppose now that we have a family of weakly geometric rough paths
$X\left(\varepsilon\right)$ with $\varepsilon\in\left[0,1\right]$
with associated paths $x\left(\varepsilon\right)$ taking values in
$\mathbb{R}^{d}$. We make the following assumptions:

\label{con:Integration}Let $X\left(\varepsilon\right)$ for $\varepsilon\in\left[0,1\right]$
be a path in $WG\Omega_{p}\left(\mathbb{R}^{d}\right)$ such that
there exists a non-negative, $0$ on the diagonal, super additive
function $\omega:\Delta_{T}\rightarrow\mathbb{R}$ satisfying\begin{align*}
\left|X_{s,t}^{i}\left(\varepsilon\right)\right|\leq & C\omega\left(s,t\right)^{\frac{i}{p}}\end{align*}
and\begin{align*}
\left|X_{s,t}^{i}\left(\varepsilon\right)-X_{s,t}^{i}\left(\tilde{\varepsilon}\right)\right|\leq & C\left|\varepsilon-\tilde{\varepsilon}\right|\omega\left(s,t\right)^{\frac{i}{p}}.\end{align*}

Given these conditions, we reformulate in terms of $\frac{1}{p}$
H\"{o}lder paths. If none of the paths are constant over any interval,
then we can define the bijection $\tau\left(t\right)=\omega\left(0,t\right)\frac{T}{\omega\left(0,T\right)}$
and reparameterize so that\begin{align}
\left|X_{\tau^{-1}\left(s\right),\tau^{-1}\left(t\right)}^{i}\right|\leq & C\omega\left(\tau^{-1}\left(s\right),\tau^{-1}\left(t\right)\right)^{\frac{i}{p}}\nonumber \\
\leq & C\left[\omega\left(\tau^{-1}\left(0\right),\tau^{-1}\left(t\right)\right)-\omega\left(\tau^{-1}\left(0\right),\tau^{-1}\left(s\right)\right)\right]^{\frac{i}{p}}\nonumber \\
= & C\frac{\omega\left(0,T\right)^{\frac{i}{p}}}{T}\left(t-s\right)^{\frac{i}{p}}.\label{eq:HolderT}\end{align}
Similarly\begin{align}
\left|X_{\tau^{-1}\left(s\right),\tau^{-1}\left(t\right)}^{i}\left(\varepsilon\right)-X_{\tau^{-1}\left(s\right),\tau^{-1}\left(t\right)}^{i}\left(\tilde{\varepsilon}\right)\right|\nonumber \\
\leq C\left|\varepsilon-\tilde{\varepsilon}\right|\frac{\omega\left(0,T\right)^{\frac{i}{p}}}{T}\left(t-s\right)^{\frac{i}{p}}.\label{eq:LipEpsilon}\end{align}
 As a result, we will assume for the rest of the section that we have
reparameterized so that the H\"{o}lder estimates in (\ref{eq:HolderT})
and (\ref{eq:LipEpsilon}) hold, i.e. when we write $X_{s,t}\left(\varepsilon\right)$
it is understood as the reparameterized $X_{\tau^{-1}\left(s\right),\tau^{-1}\left(\tau\right)}$.

We construct $Z\left(\varepsilon,\delta\right)$ in the following
manner. For each $\varepsilon$ and $\delta$, we define a choice
of area element over dyadic intervals associated to the path $\left(x\left(\varepsilon\right),x\left(\delta\right)\right)$
in $\mathbb{R}^{2d}$. We then show that over these dyadic points
the area element satisfies an estimate of the same type as (\ref{eq:LipEpsilon}).
Finally we show that the area element can be extended to all intervals
such that the estimate still holds. The constructed $Z\left(\varepsilon,\delta\right)$
will then be continuous as a map from $\left[0,1\right]^{2}\rightarrow WG\Omega_{p}\left(\mathbb{R}^{2d}\right)$
so we can define the integral of the constituent projection $\pi_{1,2}\left(Z\left(\varepsilon,\delta\right)\right)$.

From the definition of the group imbedding of weakly geometric rough
paths into $T^{2}\left(\mathbb{R}^{2d}\right)$ we get in the original
formulation that $Z\left(\varepsilon,\delta\right)^{2}$ should be
defined by \begin{align*}
Z\left(\varepsilon,\delta\right)^{2}= & \left(\begin{array}{cc}
\frac{1}{2}x^{i}\left(\varepsilon\right)x^{j}\left(\varepsilon\right)+A^{ij}\left(\varepsilon\right) & \frac{1}{2}x^{i}\left(\varepsilon\right)x^{j}\left(\delta\right)+A^{ij}\left(\varepsilon,\delta\right)\\
\frac{1}{2}x^{i}\left(\delta\right)x^{j}\left(\varepsilon\right)+A^{ij}\left(\varepsilon,\delta\right) & \frac{1}{2}x^{i}\left(\delta\right)x^{j}\left(\delta\right)+A^{ij}\left(\delta\right)\end{array}\right)\end{align*}
where in the $1,1$ entry of the block matrix, $i,j\in\left\{ 1,\ldots,d\right\} $,
in the $1,2$ entry $i\in\left\{ 1,\ldots,d\right\} $ $j\in\left\{ d+1,\ldots,2d\right\} ,$
etc. Where $A_{s,t}^{ij}$ is interpreted as the area in the $\left(i,j\right)$
plane bounded by the curve and its chord between $s$ and $t$. The
case of $A^{ij}$ corresponding to the $1,2$ entry in the block matrix
for $Z$ is the only one treated as precisely the same arguments are
used in the other terms. In order to be multiplicative, the area elements
must satisfy \begin{align}
A_{s,u}^{ij}= & A_{s,t}^{ij}+A_{tu}^{ij}+\frac{1}{2}\left(x_{s,t}^{i}\left(\varepsilon\right)x_{t,u}^{j}\left(\delta\right)-x_{s,t}^{j}\left(\delta\right)x_{t,u}^{i}\left(\varepsilon\right)\right).\label{eq:arealem}\end{align}

\begin{lem}
If $X\left(\varepsilon\right)$ satisfies condition (\ref{con:Integration})
and is also assumed to be reparameterized then there exists a map
$A$ taking the dyadic points of the simplex $\Delta_{T}$ to the
set of antisymmetric $2$ tensors over $\mathbb{R}^{2d}$ satisfying
(\ref{eq:arealem}) such that \begin{align*}
\left|A_{\frac{k}{2^{n}},\frac{k+1}{2^{n}}}^{ij}\left(\varepsilon,\delta\right)-A_{\frac{k}{2^{n}},\frac{k+1}{2^{n}}}^{ij}\left(\tilde{\varepsilon},\tilde{\delta}\right)\right|\leq & \frac{1}{2}C^{2}\frac{\omega\left(0,T\right)^{\frac{2}{p}}}{T^{2}}\left[\sum_{l=0}^{n-1}2^{\frac{l\left(2-p\right)}{p}}\right]\\
 & \times\left(\left|\varepsilon-\tilde{\varepsilon}\right|+\left|\delta-\tilde{\delta}\right|\right)2^{-\frac{2n}{p}}.\end{align*}
\end{lem}
\begin{proof}
For each $\left(\varepsilon,\delta\right)$ we can define $A^{ij}\left(\varepsilon,\delta\right)$
as follows. Set $A_{0,1}^{ij}\left(\varepsilon,\delta\right)=C$ then
suppose we have defined $A_{\frac{k}{2^{n}},\frac{k+1}{2^{n}}}^{ij}\left(\varepsilon,\delta\right)$.
We then define $A^{ij}$ on the next finest partition by setting the
areas over each of the two halves of the previous partition to be
equal. In other words, we set $A^{ij}$ from an old partition point
to the point added by the finer partition equal to $A^{ij}$ over
the added point to the next old point, i.e. $A_{\frac{2k}{2^{n+1}},\frac{2k+1}{2^{n+1}}}^{ij}\left(\varepsilon,\delta\right)=A_{\frac{2k+1}{2^{n+1}},\frac{2k+2}{2^{n+1}}}^{ij}\left(\varepsilon,\delta\right)$.
Next, as we want the final product to satisfy equation (\ref{eq:arealem})
we set \begin{align*}
A_{\frac{k}{2^{n}},\frac{k+1}{2^{n}}}^{ij}\left(\varepsilon,\delta\right)= & 2A_{\frac{2k}{2^{n+1}},\frac{2k+1}{2^{n+1}}}^{ij}\left(\varepsilon,\delta\right)\\
 & +\frac{1}{2}\left(x_{\frac{2k}{2^{n+1}},\frac{2k+1}{2^{n+1}}}^{i}\left(\varepsilon\right)x_{\frac{2k+1}{2^{n+1}},\frac{k+1}{2^{n}}}^{j}\left(\delta\right)-x_{\frac{2k}{2^{n+1}},\frac{2k+1}{2^{n+1}}}^{j}\left(\delta\right)x_{\frac{2k+1}{2^{n+1}},\frac{k+1}{2^{n}}}^{i}\left(\varepsilon\right)\right)\end{align*}
 so that\begin{align*}
A_{\frac{2k}{2^{n+1}},\frac{2k+1}{2^{n+1}}}^{ij}\left(\varepsilon,\delta\right)= & A_{\frac{2k+1}{2^{n+1}},\frac{2k+2}{2^{n+1}}}^{ij}\left(\varepsilon,\delta\right)\\
= & \frac{1}{2}A_{\frac{k}{2^{n}},\frac{k+1}{2^{n}}}^{ij}\left(\varepsilon,\delta\right)\\
 & -\frac{1}{4}\left(x_{\frac{2k}{2^{n+1}},\frac{2k+1}{2^{n+1}}}^{i}\left(\varepsilon\right)x_{\frac{2k+1}{2^{n+1}},\frac{k+1}{2^{n}}}^{j}\left(\delta\right)\right.\\
 & -\left.x_{\frac{2k}{2^{n+1}},\frac{2k+1}{2^{n+1}}}^{j}\left(\delta\right)x_{\frac{2k+1}{2^{n+1}},\frac{k+1}{2^{n}}}^{i}\left(\varepsilon\right)\right).\end{align*}
 Then, through the use of induction, we explicitly see what the area
is over each of the dyadic points for any level. Indeed\begin{align*}
A_{0,\frac{1}{2}}^{i,j}\left(\varepsilon,\delta\right)= & \frac{1}{2}C-\frac{1}{4}\left(x_{0,\frac{1}{2}}^{i}\left(\varepsilon\right)x_{\frac{1}{2},1}^{j}\left(\delta\right)-x_{0,\frac{1}{2}}^{j}\left(\delta\right)x_{\frac{1}{2},1}^{i}\left(\varepsilon\right)\right)\end{align*}
and \begin{align*}
A_{0,\frac{1}{4}}^{i,j}\left(\varepsilon,\delta\right)= & \frac{1}{4}C-\frac{1}{4}\left[\frac{1}{2}\left(x_{0,\frac{1}{2}}^{i}\left(\varepsilon\right)x_{\frac{1}{2},1}^{j}\left(\delta\right)-x_{0,\frac{1}{2}}^{j}\left(\delta\right)x_{\frac{1}{2},1}^{i}\left(\varepsilon\right)\right)\right.\\
 & \left.+\left(x_{0,\frac{1}{4}}^{i}\left(\varepsilon\right)x_{\frac{1}{4},\frac{1}{2}}^{j}\left(\delta\right)-x_{0,\frac{1}{4}}^{j}\left(\delta\right)x_{\frac{1}{4},\frac{1}{2}}^{i}\left(\varepsilon\right)\right)\right]\end{align*}
so that\begin{align*}
A_{\frac{k}{2^{n}},\frac{k+1}{2^{n}}}^{i,j}\left(\varepsilon,\delta\right)= & \frac{1}{2^{n}}C\\
 & -\frac{1}{4}\sum_{l=0}^{n-1}2^{-l}\left(x_{\frac{2k}{2^{n-l}},\frac{2k+1}{2^{n-l}}}^{i}\left(\varepsilon\right)x_{\frac{2k+1}{2^{n-l}},\frac{2k+2}{2^{n-l}}}^{j}\left(\delta\right)\right.\\
 & -\left.x_{\frac{2k}{2^{n-l}},\frac{2k+1}{2^{n-l}}}^{j}\left(\delta\right)x_{\frac{2k+1}{2^{n-l}},\frac{2k+2}{2^{n-l}}}^{i}\left(\varepsilon\right)\right),\end{align*}
for all $n\geq1$. Therefore, \begin{align*}
 & A_{\frac{k}{2^{n}},\frac{k+1}{2^{n}}}^{ij}\left(\varepsilon,\delta\right)-A_{\frac{k}{2^{n}},\frac{k+1}{2^{n}}}^{ij}\left(\tilde{\varepsilon},\tilde{\delta}\right)\\
 & =-\frac{1}{4}\sum_{l=0}^{n-1}2^{-l}\left(\left(x_{\frac{2k}{2^{n-l}},\frac{2k+1}{2^{n-l}}}^{i}\left(\varepsilon\right)x_{\frac{2k+1}{2^{n-l}},\frac{2k+2}{2^{n-l}}}^{j}\left(\delta\right)-x_{\frac{2k}{2^{n-l}},\frac{2k+1}{2^{n-l}}}^{j}\left(\delta\right)x_{\frac{2k+1}{2^{n-l}},\frac{2k+2}{2^{n-l}}}^{i}\left(\varepsilon\right)\right)\right.\\
 & \left.-\left(x_{\frac{2k}{2^{n-l}},\frac{2k+1}{2^{n-l}}}^{i}\left(\tilde{\varepsilon}\right)x_{\frac{2k+1}{2^{n-l}},\frac{2k+2}{2^{n-l}}}^{j}\left(\tilde{\delta}\right)-x_{\frac{2k}{2^{n-l}},\frac{2k+1}{2^{n-l}}}^{j}\left(\tilde{\delta}\right)x_{\frac{2k+1}{2^{n-l}},\frac{2k+2}{2^{n-l}}}^{i}\left(\tilde{\varepsilon}\right)\right)\right)\end{align*}
and\begin{align*}
 & \left|A_{\frac{k}{2^{n}},\frac{k+1}{2^{n}}}^{ij}\left(\varepsilon,\delta\right)-A_{\frac{k}{2^{n}},\frac{k+1}{2^{n}}}^{ij}\left(\tilde{\varepsilon},\tilde{\delta}\right)\right|\\
 & \leq\frac{1}{4}\left|\sum_{l=0}^{n-1}2^{-l}\left[\left(x_{\frac{2k}{2^{n-l}},\frac{2k+1}{2^{n-l}}}^{i}\left(\varepsilon\right)x_{\frac{2k+1}{2^{n-l}},\frac{2k+2}{2^{n-l}}}^{j}\left(\delta\right)-x_{\frac{2k}{2^{n-l}},\frac{2k+1}{2^{n-l}}}^{j}\left(\delta\right)x_{\frac{2k+1}{2^{n-l}},\frac{2k+2}{2^{n-l}}}^{i}\left(\varepsilon\right)\right)\right.\right.\\
 & \left.\left.-\left(x_{\frac{2k}{2^{n-l}},\frac{2k+1}{2^{n-l}}}^{i}\left(\tilde{\varepsilon}\right)x_{\frac{2k+1}{2^{n-l}},\frac{2k+2}{2^{n-l}}}^{j}\left(\tilde{\delta}\right)-x_{\frac{2k}{2^{n-l}},\frac{2k+1}{2^{n-l}}}^{j}\left(\tilde{\delta}\right)x_{\frac{2k+1}{2^{n-l}},\frac{2k+2}{2^{n-l}}}^{i}\left(\tilde{\varepsilon}\right)\right)\right]\right|.\end{align*}
From here we can add and subtract\[
x_{\frac{2k}{2^{n-l}},\frac{2k+1}{2^{n-l}}}^{i}\left(\varepsilon\right)x_{\frac{2k+1}{2^{n-l}},\frac{2k+2}{2^{n-l}}}^{j}\left(\tilde{\delta}\right)\]
and\[
x_{\frac{2k}{2^{n-l}},\frac{2k+1}{2^{n-l}}}^{j}\left(\delta\right)x_{\frac{2k+1}{2^{n-l}},\frac{2k+2}{2^{n-l}}}^{i}\left(\tilde{\varepsilon}\right)\]
 to establish that the term in square brackets above is equal to \begin{align*}
 & \left[\left(x_{\frac{2k}{2^{n-l}},\frac{2k+1}{2^{n-l}}}^{i}\left(\varepsilon\right)\left[x_{\frac{2k+1}{2^{n-l}},\frac{2k+2}{2^{n-l}}}^{j}\left(\delta\right)-x_{\frac{2k+1}{2^{n-l}},\frac{2k+2}{2^{n-l}}}^{j}\left(\tilde{\delta}\right)\right]\right.\right.\\
 & -\left.x_{\frac{2k}{2^{n-l}},\frac{2k+1}{2^{n-l}}}^{j}\left(\delta\right)\left[x_{\frac{2k+1}{2^{n-l}},\frac{2k+2}{2^{n-l}}}^{i}\left(\varepsilon\right)-x_{\frac{2k+1}{2^{n-l}},\frac{2k+2}{2^{n-l}}}^{i}\left(\tilde{\varepsilon}\right)\right]\right)\\
 & -\left(\left[x_{\frac{2k}{2^{n-l}},\frac{2k+1}{2^{n-l}}}^{i}\left(\tilde{\varepsilon}\right)-x_{\frac{2k}{2^{n-l}},\frac{2k+1}{2^{n-l}}}^{i}\left(\varepsilon\right)\right]x_{\frac{2k+1}{2^{n-l}},\frac{2k+2}{2^{n-l}}}^{j}\left(\tilde{\delta}\right)\right.\\
 & \left.\left.\left[x_{\frac{2k}{2^{n-l}},\frac{2k+1}{2^{n-l}}}^{j}\left(\delta\right)-x_{\frac{2k}{2^{n-l}},\frac{2k+1}{2^{n-l}}}^{j}\left(\tilde{\delta}\right)\right]x_{\frac{2k+1}{2^{n-l}},\frac{2k+2}{2^{n-l}}}^{i}\left(\tilde{\varepsilon}\right)\right)\right].\end{align*}
Hence by using our conditions, we have\begin{align*}
\left|A_{\frac{k}{2^{n}},\frac{k+1}{2^{n}}}^{ij}\left(\varepsilon,\delta\right)-A_{\frac{k}{2^{n}},\frac{k+1}{2^{n}}}^{ij}\left(\tilde{\varepsilon},\tilde{\delta}\right)\right|\leq & \frac{1}{2}C^{2}\frac{\omega\left(0,T\right)^{\frac{2}{p}}}{T^{2}}\left[\sum_{l=0}^{n-1}2^{\frac{l\left(2-p\right)}{p}}\right]\\
 & \times\left(\left|\varepsilon-\tilde{\varepsilon}\right|+\left|\delta-\tilde{\delta}\right|\right)2^{-\frac{2n}{p}}.\end{align*}
 
\end{proof}
Now, since we currently have only the extension defined for dyadic
points of time, let us show that we can extend the area element to
all points of the simplex $\Delta_{T}$ such that the H\"{o}lder
estimate still holds. 
\begin{lem}
There exists a unique $\hat{A}^{ij}$ defined on all points of $\Delta_{T}$
which on dyadic points coincides with $A^{ij}$ defined above such
that \[
\left|\hat{A}_{s,t}^{ij}\left(\varepsilon,\delta\right)-\hat{A}_{s,t}^{ij}\left(\tilde{\varepsilon},\tilde{\delta}\right)\right|\leq\tilde{C}\left(\left|\varepsilon-\tilde{\varepsilon}\right|+\left|\delta-\tilde{\delta}\right|\right)\left(t-s\right)^{\frac{2}{p}}\]
for all $s,t\in\left[0,1\right]$ for a constant $\tilde{C}$ depending
on $p$.\end{lem}
\begin{proof}
This proof follows the proof of lemma 2 in the Lyons-Victoir extension
paper \cite{LyonsVicExten}. We have established that, after reparameterization,
the area elements given above on dyadic points satisfy \[
\left|A_{\frac{k}{2^{n}},\frac{k+1}{2^{n}}}^{ij}\left(\varepsilon,\delta\right)-A_{\frac{k}{2^{n}},\frac{k+1}{2^{n}}}^{ij}\left(\tilde{\varepsilon},\tilde{\delta}\right)\right|\leq4\hat{C}\left(\left|\varepsilon-\tilde{\varepsilon}\right|+\left|\delta-\tilde{\delta}\right|\right)2^{-\frac{2n}{p}}\]
where $\hat{C}=\frac{1}{2}C^{2}\frac{\omega\left(0,T\right)^{\frac{2}{p}}}{T^{2}}\left[\sum_{l=0}^{\infty}2^{\frac{l\left(2-p\right)}{p}}\right]$,
since $\sum_{l=0}^{n-1}2^{\frac{l\left(2-p\right)}{2}}$ converges
as $2<p<3$. The first step is to prove the second inequality when
$s,t$ are dyadic points from the same level of fineness but not necessarily
consecutive points. Let $D_{m}=\cup_{k=0}^{2^{m}}\frac{k}{2^{m}}$
and consider all $s,t\in D_{m}$ such that $0<t-s<2^{-r}$ for some
fixed integer $r$. We want to show by induction that \begin{align*}
\left|A_{s,t}^{ij}\left(\varepsilon,\delta\right)-A_{s,t}^{ij}\left(\tilde{\varepsilon},\tilde{\delta}\right)\right| & \leq4\hat{C}\left(\left|\varepsilon-\tilde{\varepsilon}\right|+\left|\delta-\tilde{\delta}\right|\right)\sum_{k=r+1}^{m}2^{-\frac{2k}{p}}.\end{align*}
For the case when $m=r+1$ the above is identical to the condition.
For the inductive step, assume the statement is true for $m=r+1,\ldots,M-1$
and consider $s,t\in D_{M}$ with $0<t-s<2^{-r}$. Define two points
$s_{1}$ and $t_{1}$from the next coarsest level of dyadic points
which are nonetheless adjacent to $s$ and $t$, i.e. $s_{1}=\min\left\{ u\in D_{M-1}:u\geq s\right\} $
and $t_{1}:\max=\left\{ u\in D_{M-1}:u\leq t\right\} $. Then \begin{align*}
\left|A_{s,s_{1}}^{ij}\left(\varepsilon,\delta\right)-A_{s,s_{1}}^{ij}\left(\tilde{\varepsilon},\tilde{\delta}\right)\right|\leq & \hat{C}\left(\left|\varepsilon-\tilde{\varepsilon}\right|+\left|\delta-\tilde{\delta}\right|\right)2^{-\frac{2M}{p}}\\
\left|A_{t_{1},t}^{ij}\left(\varepsilon,\delta\right)-A_{t_{1},t}^{ij}\left(\tilde{\varepsilon},\tilde{\delta}\right)\right|\leq & \hat{C}\left(\left|\varepsilon-\tilde{\varepsilon}\right|+\left|\delta-\tilde{\delta}\right|\right)2^{-\frac{2M}{p}}.\end{align*}
Now, the area elements $A^{ij}$ corresponding to the $1,2$ entry
in the block matrix for $Z$ were constructed to satisfy \begin{align*}
A_{s,t}^{ij}\left(\varepsilon,\delta\right)= & A_{s,s_{1}}^{ij}\left(\varepsilon,\delta\right)+A_{s_{1},t}^{ij}\left(\varepsilon,\delta\right)+\frac{1}{2}\left(x_{s,s_{1}}^{i}\left(\varepsilon\right)x_{s_{1},t}^{j}\left(\delta\right)-x_{s,s_{1}}^{j}\left(\varepsilon\right)x_{s,t}^{i}\left(\delta\right)\right)\end{align*}
so that \begin{align*}
A_{s,t}^{ij}\left(\varepsilon,\delta\right)-A_{s,t}^{ij}\left(\tilde{\varepsilon},\tilde{\delta}\right)= & A_{s,s_{1}}^{ij}\left(\varepsilon,\delta\right)-A_{s,s_{1}}^{ij}\left(\tilde{\varepsilon},\tilde{\delta}\right)+A_{s_{1},t_{1}}^{ij}\left(\varepsilon,\delta\right)-A_{s_{1},t_{1}}^{ij}\left(\tilde{\varepsilon},\tilde{\delta}\right)\\
 & +A_{t_{1},t}^{ij}\left(\varepsilon,\delta\right)-A_{t_{1},t}^{ij}\left(\tilde{\varepsilon},\tilde{\delta}\right)\\
 & +\frac{1}{2}\left(x_{s,s_{1}}^{i}\left(\varepsilon\right)x_{s_{1},t}^{j}\left(\delta\right)-x_{s,s_{1}}^{j}\left(\varepsilon\right)x_{s,t}^{i}\left(\delta\right)\right)\\
 & -\frac{1}{2}\left(x_{s,s_{1}}^{i}\left(\tilde{\varepsilon}\right)x_{s_{1},t}^{j}\left(\tilde{\delta}\right)-x_{s,s_{1}}^{j}\left(\tilde{\varepsilon}\right)x_{s,t}^{i}\left(\tilde{\delta}\right)\right)\\
 & +\frac{1}{2}\left(x_{s_{1},t_{1}}^{i}\left(\varepsilon\right)x_{t_{1},t}^{j}\left(\delta\right)-x_{s_{1},t_{1}}^{j}\left(\delta\right)x_{t_{1},t}^{i}\left(\varepsilon\right)\right)\\
 & -\frac{1}{2}\left(x_{s_{1},t_{1}}^{i}\left(\tilde{\varepsilon}\right)x_{t_{1},t}^{j}\left(\tilde{\delta}\right)-x_{s_{1},t_{1}}^{j}\left(\tilde{\delta}\right)x_{t_{1},t}^{i}\left(\tilde{\varepsilon}\right)\right).\end{align*}
Applying the inductive step and adding and subtracting\begin{align*}
 & x_{s,s_{1}}^{i}\left(\varepsilon\right)x_{s_{1},t}^{j}\left(\tilde{\delta}\right),\\
 & x_{s,s_{1}}^{j}\left(\delta\right)x_{s_{1},t}^{i}\left(\tilde{\varepsilon}\right),\\
 & x_{s_{1},t_{1}}^{i}\left(\varepsilon\right)x_{t_{1},t}^{j}\left(\tilde{\delta}\right),\end{align*}
and\begin{align*}
 & x_{s_{1},t_{1}}^{j}\left(\delta\right)x_{t_{1},t}^{i}\left(\tilde{\varepsilon}\right)\end{align*}
we get \begin{align*}
\left|A_{s,t}^{ij}\left(\varepsilon,\delta\right)-A_{s,t}^{ij}\left(\tilde{\varepsilon},\tilde{\delta}\right)\right|\leq & 2\hat{C}\left(\left|\varepsilon-\tilde{\varepsilon}\right|+\left|\delta-\tilde{\delta}\right|\right)2^{-\frac{2M}{p}}\\
 & +\hat{C}\left(\left|\varepsilon-\tilde{\varepsilon}\right|+\left|\delta-\tilde{\delta}\right|\right)\sum_{k=r+1}^{M-1}2^{-\frac{2k}{p}}\\
 & +\frac{1}{2}\left|x_{s,s_{1}}^{i}\left(\varepsilon\right)\left[x_{s_{1},t}^{j}\left(\delta\right)-x_{s_{1},t}^{j}\left(\tilde{\delta}\right)\right]\right|\\
 & +\frac{1}{2}\left|x_{s,s_{1}}^{j}\left(\delta\right)\left[x_{s_{1},t}^{i}\left(\varepsilon\right)-x_{s_{1},t}^{i}\left(\tilde{\varepsilon}\right)\right]\right|\\
 & +\frac{1}{2}\left|\left[x_{s,s_{1}}^{i}\left(\tilde{\varepsilon}\right)-x_{s,s_{1}}^{i}\left(\varepsilon\right)\right]x_{s_{1},t}^{j}\left(\tilde{\delta}\right)\right|\\
 & +\frac{1}{2}\left|\left[x_{s,s_{1}}^{j}\left(\tilde{\delta}\right)-x_{s,s_{1}}^{j}\left(\delta\right)\right]x_{s,t}^{i}\left(\tilde{\varepsilon}\right)\right|\\
 & +\frac{1}{2}\left|x_{s_{1},t_{1}}^{i}\left(\varepsilon\right)\left[x_{t_{1},t}^{j}\left(\delta\right)-x_{t_{1},t}^{j}\left(\tilde{\delta}\right)\right]\right|\\
 & +\frac{1}{2}\left|x_{s_{1},t_{1}}^{j}\left(\delta\right)\left[x_{t_{1},t}^{i}\left(\varepsilon\right)-x_{t_{1},t}^{i}\left(\tilde{\varepsilon}\right)\right]\right|\\
 & +\frac{1}{2}\left|\left[x_{s_{1},t_{1}}^{i}\left(\tilde{\varepsilon}\right)-x_{s_{1},t_{1}}^{i}\left(\varepsilon\right)\right]x_{t_{1},t}^{j}\left(\tilde{\delta}\right)\right|\\
 & +\frac{1}{2}\left|\left[x_{s_{1},t_{1}}^{j}\left(\tilde{\delta}\right)-x_{s_{1},t_{1}}^{j}\left(\delta\right)\right]x_{t_{1},t}^{i}\left(\tilde{\varepsilon}\right)\right|.\end{align*}
Finally, we apply the condition to get \begin{align*}
\left|A_{s,t}^{ij}\left(\varepsilon,\delta\right)-A_{s,t}^{ij}\left(\tilde{\varepsilon},\tilde{\delta}\right)\right|\leq & 2\hat{C}\left(\left|\varepsilon-\tilde{\varepsilon}\right|+\left|\delta-\tilde{\delta}\right|\right)2^{-\frac{2M}{p}}\\
 & +4\hat{C}\left(\left|\varepsilon-\tilde{\varepsilon}\right|+\left|\delta-\tilde{\delta}\right|\right)\sum_{k=r+1}^{M-1}2^{-\frac{2k}{p}}\\
 & +2\hat{C}\left(\left|\varepsilon-\tilde{\varepsilon}\right|+\left|\delta-\tilde{\delta}\right|\right)2^{-\frac{2M}{p}}\end{align*}
so that \[
\left|A_{s,t}^{ij}\left(\varepsilon,\delta\right)-A_{s,t}^{ij}\left(\tilde{\varepsilon},\tilde{\delta}\right)\right|\leq4\hat{C}\left(\left|\varepsilon-\tilde{\varepsilon}\right|+\left|\delta-\tilde{\delta}\right|\right)\sum_{k=r+1}^{M}2^{-\frac{2k}{p}}\]
as required. This works because the sum $\sum_{l=0}^{\infty}2^{\frac{l\left(2-p\right)}{p}}$
is larger than $1$. For all points $s,t\in\cup_{m}D_{m}$ such that
$2^{-\left(r+1\right)}<t-s<2^{-r}$\begin{align*}
\left|A_{s,t}^{ij}\left(\varepsilon,\delta\right)-A_{s,t}^{ij}\left(\tilde{\varepsilon},\tilde{\delta}\right)\right|\leq & 4\hat{C}\left(\left|\varepsilon-\tilde{\varepsilon}\right|+\left|\delta-\tilde{\delta}\right|\right)\sum_{k=r+1}^{\infty}2^{-\frac{2k}{p}}\\
= & 4\hat{C}\left(\left|\varepsilon-\tilde{\varepsilon}\right|+\left|\delta-\tilde{\delta}\right|\right)2^{\frac{-2\left(r+1\right)}{p}}\sum_{k=0}^{\infty}2^{-\frac{2k}{p}}\\
\leq & \tilde{C}\left(\left|\varepsilon-\tilde{\varepsilon}\right|+\left|\delta-\tilde{\delta}\right|\right)\left(t-s\right)^{-\frac{2}{p}}.\end{align*}
Next, using the fact that for any real number $t$ , $\frac{\left\lfloor 2^{r}t\right\rfloor }{2^{r}}\rightarrow t$
as $r\rightarrow\infty$, we define for arbitrary $s,t\in\left[0,T\right]$
\begin{align*}
\hat{A}_{s,t}^{ij}\left(\varepsilon,\delta\right)= & \lim_{r\rightarrow\infty}A_{\frac{\left\lfloor 2^{r}s\right\rfloor }{2^{r}},\frac{\left\lfloor 2^{r}t\right\rfloor }{2^{r}}}^{ij}\left(\varepsilon,\delta\right)\end{align*}
and so, by continuity of the norm, the required estimate holds.
\end{proof}
From this we can define an extension $Z\left(\varepsilon,\delta\right)$
of $\left(X\left(\varepsilon\right),X\left(\delta\right)\right)$
where\begin{align*}
Z\left(\varepsilon,\delta\right)^{2}= & \left(\begin{array}{cc}
\frac{1}{2}x^{i}\left(\varepsilon\right)x^{j}\left(\varepsilon\right)+\hat{A}^{ij}\left(\varepsilon\right) & \frac{1}{2}x^{i}\left(\varepsilon\right)x^{j}\left(\delta\right)+\hat{A}^{ij}\left(\varepsilon,\delta\right)\\
\frac{1}{2}x^{i}\left(\delta\right)x^{j}\left(\varepsilon\right)+\hat{A}^{ij}\left(\varepsilon,\delta\right) & \frac{1}{2}x^{i}\left(\delta\right)x^{j}\left(\delta\right)+\hat{A}^{ij}\left(\delta\right)\end{array}\right).\end{align*}
This extension then satisfies \begin{align*}
\left|Z^{2}\left(\varepsilon,\delta\right)_{s,t}-Z^{2}\left(\tilde{\varepsilon},\tilde{\delta}\right)_{s,t}\right|\leq & C\left(\left|\varepsilon-\tilde{\varepsilon}\right|+\left|\delta-\tilde{\delta}\right|\right)\left(t-s\right)^{-\frac{2}{p}}\end{align*}
which means we can bound $d_{p}\left(Z\left(\varepsilon,\delta\right),Z\left(\tilde{\varepsilon},\tilde{\delta}\right)\right)$by
a constant multiplied by $\left|\varepsilon-\tilde{\varepsilon}\right|+\left|\delta-\tilde{\delta}\right|$. 

Since we now know the constructed $Z\left(\varepsilon,\delta\right)$
is continuous as a map from $\left(\varepsilon,\delta\right)\rightarrow WG\Omega_{p}\left(\mathbb{R}^{d}\right)$,
it is also measurable and so integration makes sense. 
\begin{defn}
\label{def:Integral}Given a path $X\left(\varepsilon\right)$ in
$W\Omega_{p}\left(\mathbb{R}^{d}\right)$ satisfying condition \ref{con:Integration}
and a measure $\mu$ supported on $\left[0,1\right]$, we define its
integral $X^{\mu}$ by\begin{align*}
X_{s,t}^{\mu,1}= & \int_{0}^{1}X_{s,t}^{1}\left(\varepsilon\right)\mu\left(d\varepsilon\right)\\
X_{s,t}^{\mu,2}= & \int_{0}^{1}\int_{0}^{1}\pi_{1,2}\left(Z\left(\varepsilon,\delta\right)\right)_{s,t}\mu\left(d\varepsilon\right)\mu\left(d\delta\right)\end{align*}
where $Z$ is defined by the extension constructed above.
\end{defn}
In the following proposition we show the integrated path is still
a rough path.
\begin{prop}
The object $X^{\mu}$ in definition \ref{def:Integral} is multiplicative.\end{prop}
\begin{proof}
We have \begin{align*}
\left(X_{s,t}^{\mu}\otimes X_{t,u}^{\mu}\right)^{1}= & \int_{0}^{1}X_{s,t}^{1}\left(\varepsilon\right)\mu\left(d\varepsilon\right)+\int_{0}^{1}X_{t,u}^{1}\left(\varepsilon\right)\mu\left(d\varepsilon\right)\\
= & \int_{0}^{1}\left[X_{s,t}^{1}\left(\varepsilon\right)+X_{t,u}^{1}\left(\varepsilon\right)\right]\mu\left(d\varepsilon\right)\\
= & \int_{0}^{1}X_{s,u}^{1}\left(\varepsilon\right)\mu\left(d\varepsilon\right)\\
= & X_{s,u}^{\mu,1}\end{align*}
since $X$ is multiplicative. Also,\begin{align*}
\left(X_{s,t}^{\mu}\otimes X_{t,u}^{\mu}\right)^{2}= & \int_{0}^{1}\int_{0}^{1}\pi_{1,2}\left(Z\left(\varepsilon,\delta\right)\right)_{s,t}\mu\left(d\varepsilon\right)\mu\left(d\delta\right)\\
 & +\int_{0}^{1}\int_{0}^{1}\pi_{1,2}\left(Z\left(\varepsilon,\delta\right)\right)_{t,u}\mu\left(d\varepsilon\right)\mu\left(d\delta\right)\\
 & +\int_{0}^{1}X_{s,t}^{1}\left(\varepsilon\right)\mu\left(d\varepsilon\right)\otimes\int_{0}^{1}X_{t,u}^{1}\left(\delta\right)\mu\left(d\delta\right)\end{align*}
and since $Z\left(\varepsilon,\delta\right)$ is multiplicative \begin{align*}
\pi_{1,2}\left(Z\left(\varepsilon,\delta\right)\right)_{s,u}= & \pi_{1,2}\left(Z\right)_{s,t}+\pi_{1,2}\left(Z\right)_{t,u}+X_{s,t}^{1}\left(\varepsilon\right)\otimes X_{t,u}^{1}\left(\delta\right).\end{align*}
This together with the linearity of the integrals implies \begin{align*}
\left(X_{s,t}^{\mu}\otimes X_{t,u}^{\mu}\right)^{2}= & \int_{0}^{1}\int_{0}^{1}\pi_{1,2}\left(Z\left(\varepsilon,\delta\right)\right)_{s,u}\mu\left(d\varepsilon\right)\mu\left(d\delta\right)\\
= & X_{s,u}^{\mu,2}.\end{align*}
\end{proof}


\begin{thebibliography}{18}
\bibitem{CamMart}Cameron, R. H., Martin, W. T., Transformations of
Wiener Integrals Under Translations. \textit{Ann. of Math.} \textbf{45}
(1944), 386-396.

\bibitem{TgtProcessFlow}Cirpriano, F., Cruziro, A. B., Flows Associated
to Tangent Processes on the Wiener Space. \textit{J. Funct. Anal.}
\textbf{166} (1999), No. 2, 310-331.

\bibitem{Cruz1st}Cruzeiro, A.B., Equations Differentielles sur l'espace
de Wiener et Formules de Cameron-Martin non lineares. \textit{J. Funct.
Anal.} \textbf{54} (1983), 206-227.

\bibitem{TgtProcessDef}Cruziero, A.B., Malliavin, P., Renormalized
Differential Geometry on Path Space: Structural Equation, Curvature.
\textit{J. Funct. Anal.} \textbf{139} (1996), No. 1, 119-181.

\bibitem{DriverFlow}Driver, B., A Cameron-Martin Type Quasi-Invarience
Theorem for Brownian Motion on a Compact Riemannian Manifold. \textit{J.
Funct. Anal.} \textbf{110} (1992), 273-376.

\bibitem{DriverFlow2}Driver, B., Towards Calculus and Geometry on
Path Spaces. \textit{Proc. Sympos. Pure Math.} \textbf{57} (1995),
405-422.

\bibitem{key-5}Elworthy, D., Li, X. M., A Class of Integration by
Parts Formula in Stochastic Analysis I. in {}``It\^{o}'s Calculus
and Probability Theory'' (Ikeda, N., Watanabe, S., Fukushima, M.,
Kunia, H.,) \textit{Springer-Verlag, Tokyo,} (1996) 15-30.

\bibitem{key-7}Elworthy, D., Li, X. M., It\^{o} Maps and Analysis
on Path Spaces. \textit{Math. Zeit.} \textbf{257} (2007), No. 3, 643-706.

\bibitem{NoteOnGeoRP}Friz, P., Victoir, N., A Note on the Notion
of Geometric Rough Paths \textit{Prob. Theory Relat. Fields }\textbf{136
}(2006), 395-416.

\bibitem{RotationVar}Hu, Y., \"{U}st\"{u}nel, A. S., Zakai, M.,
Tangent Processes on Wiener Space. \textit{J. Funct. Anal.} \textbf{192}
(2002), No. 1, 234-270.

\bibitem{LyonsIbero}Lyons, T., Differential Equations Driven by Rough
Signals. \textit{Rev. Mat. Iberoamer., }\textbf{14}\textbf{\textit{
}}(1998), 215-310. 

\bibitem{VFonPathSpace}Lyons, T., Qian, Z., A Class of Vector Fields
on Path Spaces \textit{J. Funct. Anal.} \textbf{145} (1997), 205-223.

\bibitem{FlowOnRP}Lyons, T., Qian, Z., Flow Equations on Spaces of
Rough Paths. \textit{J. Funct. Anal.} \textbf{149} (1997) 135-159.

\bibitem{SysControlRP}Lyons, T., Qian, Z., System Control and Rough
Paths. 

\bibitem{LyonsVicExten}Lyons, T., Victoir, N., An Extension Theorem
to Rough Paths. \textit{Ann. Inst. H. Poincar\'{e} Anal. Non Lin\'{e}aire}
\textbf{24} (2007), No. 5, 835-847.

\bibitem{MalDerivOrig}Malliavin, P., Stochastic Calculus of Variation
and Hypoelliptic Operators. \textit{Proc. Intern. Symp. SDE, Kyoto,
19}76, Wiley, New York, (1978) 195-263. 

\bibitem{MalliNatur}Malliavin, P., Naturality of Quasi-Invarience
of Some Measures. \textit{Stochastic Analysis and Applications: Proc.
Lisbon, 1989}, \textit{Prog. Prob.} \textbf{26}, Birkhauser, Boston,
(1991) 144-154.

\bibitem{MalliLoops}Malliavin, M.-P., Malliavin, P., Integration
on Loop Groups I., Quasi Invarient Measures, Quasi Invarient Integration
on Loop Groups, \textit{J. Funct. Anal.} \textbf{93} (1990), 207-237.
\end{thebibliography}
\end{document}